\documentclass{article}

\usepackage[margin=3cm]{geometry}

\usepackage{amsmath}
\usepackage{amsfonts}
\usepackage{amsthm}
\usepackage{amssymb}

\usepackage{hyperref}
 
\usepackage{graphicx}

\usepackage[noend]{algpseudocode}
 \usepackage{enumitem}

\usepackage{tikz}
\usepackage{tikz-cd}
\usetikzlibrary{decorations.markings,positioning,arrows,matrix,hobby}

\newtheorem{theorem}{Theorem}
\newtheorem{definition}{Definition}
\newtheorem{proposition}[theorem]{Proposition}
\newtheorem{remark}{Remark}
\newtheorem{lemma}[theorem]{Lemma}

\newtheorem{corollary}[theorem]{Corollary}

\newcommand{\R}{\mathbb{R}}
\newcommand{\Z}{\mathbb{Z}}
\newcommand{\F}{\mathbb{F}}
\newcommand{\N}{\mathbb{N}}
\newcommand{\T}{\mathbb{T}}
\newcommand{\U}{\mathcal{U}}
\newcommand{\nerve}{\mathcal{N}}

\newcommand{\D}{\partial}
\newcommand{\ep}{\varepsilon}
\newcommand{\inc}{\hookrightarrow}
\newcommand{\surj}{\twoheadrightarrow}
\newcommand{\tr}{\mathbf{t}}

\DeclareMathOperator{\rank}{rank}

\DeclareMathOperator{\im}{im}

\title{Quantifying the homology of periodic cell complexes}  

\author{Adam Onus\footnote{School of Mathematical Sciences, Queen Mary University of London, UK (\href{a.onus@qmul.ac.uk}{a.onus@qmul.ac.uk}).
}~ and 
Vanessa Robins\footnote{Research School of Physics, The Australian National University, Canberra, Australia. (\href{vanessa.robins@anu.edu.au}{vanessa.robins@anu.edu.au}).
}}

\begin{document}
\maketitle

\begin{abstract}
A periodic cell complex, $K$, has a finite representation as the quotient space, $q(K)$, consisting of equivalence classes of cells identified under the translation group acting on $K$.
We study how the Betti numbers and cycles of $K$ are related to those of $q(K)$, first for the case that $K$ is a graph, and then higher-dimensional cell complexes.
When $K$ is a $d$-periodic graph, it is possible to define $\Z^d$-weights on the edges of the quotient graph and this information permits full recovery of homology generators for $K$.
The situation for higher-dimensional cell complexes is more subtle and studied in detail using the Mayer-Vietoris spectral sequence.



\end{abstract}


\section{Introduction}
\label{sec:Introduction}

Spatially periodic point patterns arise naturally as models of atomic positions in crystalline materials, and as a tractable way to simulate many interacting objects without the influence of boundary effects.   Although simulations using periodic boundary conditions treat points as located in a flat $d$-dimensional torus, the structure being modelled is really some large finite domain built from many copies of a unit cell and thus a subset of $\R^d$.  

Given the increasing usefulness of persistent homology in many application areas, particularly materials science, the following questions naturally arise.
\begin{enumerate}
\item Is the persistence diagram of an infinite crystalline structure well-defined, since the persistent homology of most crystalline structures will not be q-tame, which is usually a minimum requirement for the existence of persistence diagrams \cite{chazal2016structure}.
\item How to normalise the persistence diagram of an infinite crystalline structure so that it is independent of the unit cell used. 
\item How to approximate the persistence diagram for a large finite domain of a periodic point pattern in $\R^d$ given a periodic unit cell embedded in a flat $d$-torus.  
\end{enumerate}

Physical intuition from more familiar geometric properties suggests that we should be able to normalise the number of points in a persistence diagram by the volume of the domain in $\R^d$ to obtain a quantity that is independent of domain size.
This is exactly how we define the porosity of a material for example:  as the volume of solid matter normalised by the volume of the domain.  
However, this type of normalisation is appropriate only when working with what physicists term an \emph{extensive property}: a quantity $f$ that satisfies the properties of a \emph{valuation}, most notably the inclusion-exclusion formula $f(A\cup B)  = f(A) + f(B) - f( A \cap B)$.  
%
It is easy to show that the Betti number invariants of homology $\beta_k$, and consequently the persistence diagrams, are not valuations.  For example, two $\sqsubset$-shaped domains that overlap to form a square annulus have 
$\beta_k(\sqsubset) + \beta_k(\sqsupset) \neq \beta_k(\square) + \beta_k(^{-}_{-})$, for $k=0,1$.   

Given this fundamental obstruction to a simple normalisation procedure for persistence diagrams, the current paper begins the process of providing an answer to the above questions by studying the homology of periodic cell complexes.
We are particularly motivated by peculiar behaviour in the topology of finite presentations of these complexes, such as two disconnected interwoven graphs over cubical lattices projecting onto a single connected quotient graph as in Figure~\ref{fig:interwoven-lattices} below.
We present a detailed study of how the Betti numbers of a finite cell complex (the quotient space) are related to those of its infinite periodic cover, $K$. 

\begin{figure}[h]
\centering
\includegraphics[width=0.42\linewidth]{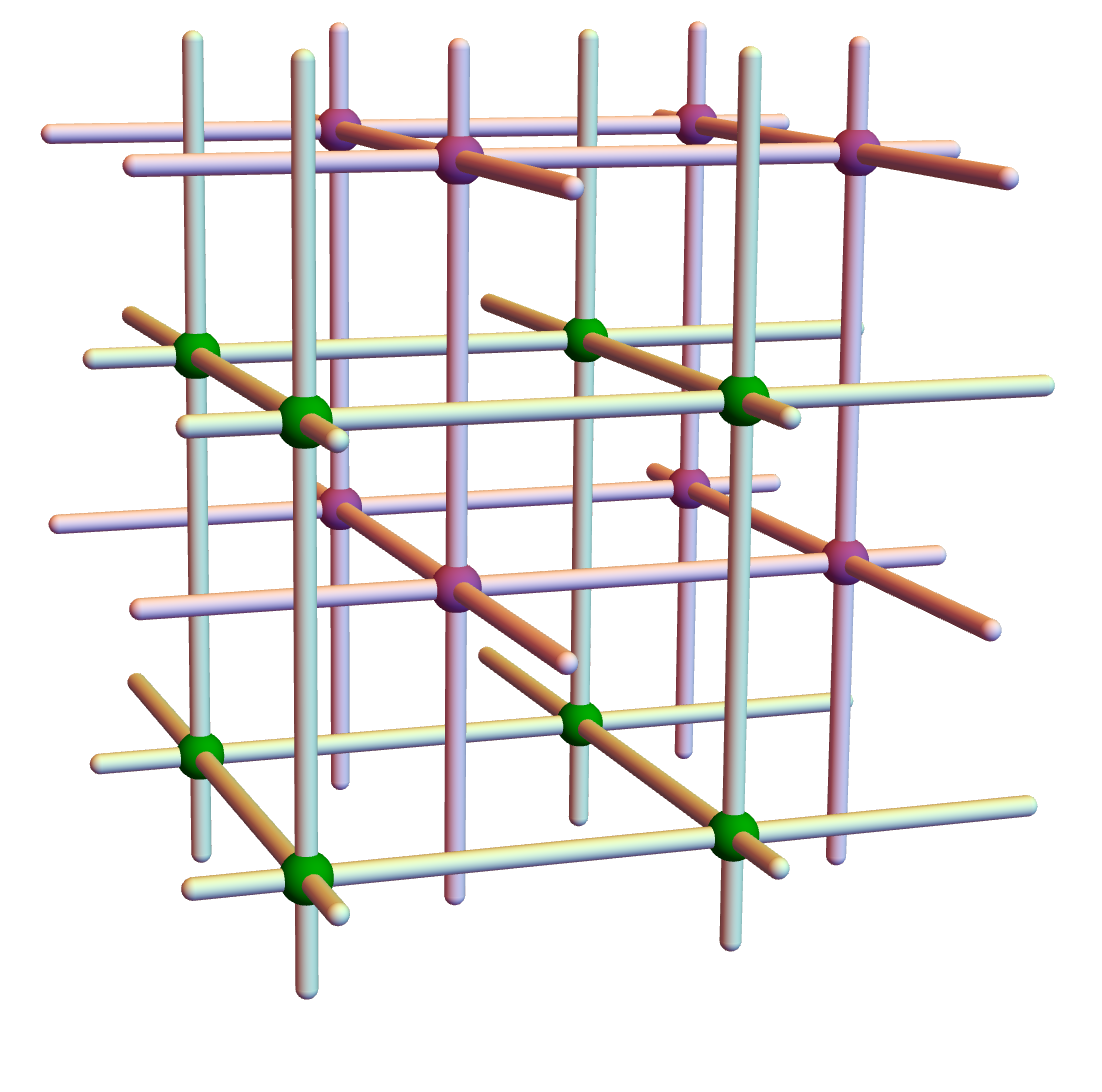}\hspace{0.4cm}
\begin{tikzpicture}
\node[shape=circle,fill,scale=0.3] (v) at (0,0) {};
\node[shape=circle,fill=white,scale=0.3] () at (-2,0) {};
\node[shape=circle,fill=white,scale=0.3] () at (0,-3.4) {};
\node[shape=circle,fill=white,scale=0.3] () at (0,2) {};
\node[shape=circle,fill=white,scale=0.3] () at (2,0) {};
\node[shape=circle,fill=white,scale=1] () at (1.2,-1.15) {$(0,1,1)$};
\node[shape=circle,fill=white,scale=1] () at (0,1.65) {$(1,0,1)$};
\node[shape=circle,fill=white,scale=1] () at (-1.3,-1.15) {$(1,1,0)$};
\draw[>=latex,->,line width=0.3mm] (v) to [out=360,in=300,looseness=105] (v);
\draw[>=latex,->,line width=0.3mm] (v) to [out=120,in=60,looseness=105] (v);
\draw[>=latex,->,line width=0.3mm] (v) to [out=240,in=180,looseness=105] (v);
\end{tikzpicture}
%
\caption{A section of two interwoven cubical lattices along integer and half integer coordinates in $\R^3$ (with connected components coloured green and purple) and its weighted quotient graph with respect to a basis of translations in the directions $t_1 = (-\frac12,\frac12,\frac12)$, $t_2 = (\frac12,-\frac12,\frac12)$ and $t_3 = (\frac12,\frac12,-\frac12)$.}
\label{fig:interwoven-lattices}
\end{figure}

The case of a periodic 1-dimensional cell-complex in $\R^d$ i.e., a \emph{periodic graph}, is considerably simpler than the general $k$-dimensional case, and has well studied representations \cite{delgado2005three,boyd2016generalized,eon2018planar}.
Results about the number of connected components ($\beta_0$) of periodic graphs can be found in computer science, electronics, crystallography and graph theory literature \cite{dicks1989groups,cohen1990recognizing,bass1993covering,sunada2012topological}.
We rephrase these results in Section~\ref{sec:Graphs} using the terminology of chain complexes and homology so that we can study the generalisation to higher dimensions.
In short, a periodic graph with undirected edges is mapped to the quotient space built from translational equivalence classes of vertices and edges (see Figure~\ref{fig:interwoven-lattices}). 
This finite quotient graph can be given $d$-dimensional vector weights that encode a translational offset between the vertex representatives at each end of the oriented edge. 
The number of components of the infinite periodic graph is then determined by comparing the span of the vector weights with the unit lattice basis for $\R^d$. 
 
In Section~\ref{sec:Complexes} we then look at $k$-dimensional periodic cell-complexes, $k\geq2$, and discuss how there is no simple method analogous to the vector weights on edges for encoding the translational offsets of boundaries of $k$-cells. 
With no simple generalisations of the formulae derived for periodic graphs, we have trouble distinguishing cycles in $K$ by studying cycles in its quotient.
We instead look at successively larger (but still periodic) finite domains, $X$, and establish a heuristic to identify \emph{toroidal}  cycles that are due to the periodic boundary conditions of $X$ and do not lift to cycles in the infinite structure $K$.
This involves application of the Mayer-Vietoris spectral sequence (MVSS), which we introduce in Section~\ref{sec:mvss}.
Using only local computations, the MVSS enables us to identify and classify true cycles of $K$ and impose an approximate lower bound for the size of $X$ required to view these features.

Our results are motivated by application to the analysis of crystal structures, where we have a periodic point cloud in space (i.e., atom positions) and a fixed cellular complex where edges represent bonds and higher dimensional cells represent higher order atomic interactions, or constant-potential surfaces.
Persistent homology is another natural tool for studying topological and geometric structure of  periodic point clouds. Although we focus on regular homology in this paper, we will note when a result can be adjusted to the case of persistent homology.  

\subsection{Related Work}
\label{sec:Related}



In Section~\ref{sec:Graphs} we define the notion of a \textit{weighted quotient graph}.
The definition we employ agrees with the ``vector method'' introduced in \cite{chung1984vector}, to which we direct the reader for more details.
In recent years, weighted quotient graphs have been studied extensively and become ubiquitous with structure classification in topological crystallography (c.f., \cite{delgado2005three,eon2018planar}), where they are known as \textit{labelled quotient graphs}).
A weighted quotient graph is also referred to as a \textit{static graph} in a discrete mathematical setting (c.f., \cite{iwano1987testing,cohen1990recognizing}) where it is used to model Very Large Scale Integration and dynamic optimisation problems (see \cite{kosaraju1988detecting,orlin1984problems}).
In this setting, the periodic graph is called the \textit{dynamic graph} constructed by interpreting edge weights as shift vectors which generate translational symmetries.

In Section~\ref{sec:Graphs} we present Theorems~\ref{thm:graphH0} and~\ref{thm:graphH1} which relate the homology of a periodic graph to properties of its weighted quotient graph.
The degree-0 result (Theorem~\ref{thm:graphH0}) was first derived by Cohen \& Megiddo in \cite{cohen1990recognizing} in the language of electronics and computer science and was independently reformulated in the language of graph theory in Corollary~1.9.3 of \cite{dicks1989groups} and Theorem~3.6 of \cite{bass1993covering} to prove an analogous result for any action on a connected graph. 
We rephrase and provide a proof using using the language of algebraic topology. 
Recent work in \cite{edelsbrunner2024merge} has also extended the approach with weights of weighted quotient graphs (which they call shift vectors) to define a degree-0 persistent homology theory for periodic spaces.
The approach of \cite{edelsbrunner2024merge}, which introduces the concept of a persistent merge tree, has a similar theme to the current paper, but is currently unable to extend to higher dimensional (persistent) homology.
Our treatment of the degree-1 homology of a periodic graph, Theorem~\ref{thm:graphH1}, uses covering space theory and builds on results from  Sunada's topological crystallography~\cite{sunada2012topological}. 

In the case of a 1-periodic cell complex $K$ (of arbitrary dimension) we may think of the map $K\to q(K)$ sending $K$ to its quotient space $q(K)$ of translational equivalence classes as being equivalent to a map $K\to \mathbb{S}^1$.
This case has been well-studied with Novikov homology \cite{novikov1991quasiperiodic} which generalises the methods of Morse theory, allowing one to explicitly calculate the homology of $K$ over the field of Laurent power series. A computer-friendly method to calculate this with Jordan blocks is described in \cite{burghelea2017topology}.
More generally, the (co)homology of the quotient space is exactly the equivariant (co)homology of $K$ with respect to the group action of translations \cite{tu2020introductory}. Recent work has looked at applications of (persistent) equivariant (co)homology to shape reconstruction and Vietoris-Rips complexes \cite{carbone2020equivariant,adams2024persistent}.

The Mayer-Vietoris spectral sequence (MVSS) was first used in topological data analysis to localise low-dimensional cycles of a topological space \cite{zomorodian2008localized}.
More recently, it has been presented primarily as a method for parallelising persistent homology calculations of large data sets (c.f. \cite{lipsky2011parallelized,boltcheva2010constructive}) and has been implemented for abstract simplicial complexes \cite{lewis2014multicore}.
In \cite{govc2018approximate} the authors use a persistence version of the MVSS to prove an approximate nerve theorem for persistence diagrams.
Computer implementations of a persistence Mayer-Vietoris spectral sequence can be found in \cite{lewis2015parallel,casas2019distributing}, although these encounter what \cite{govc2018approximate} refer to as the ``extension problem'', which we briefly discuss in Section~\ref{sec:mvss-def}.
It is also well-known that persistent homology can be calculated through the spectral sequence of a filtration (c.f. \cite{bauer2014clear,basu2017spectral}).
However, none of these applications have yet been used to study periodic spaces.


\section{Background and Notation}
\label{sec:Definitions}

This section covers basic definitions and sets up notation for the objects studied in this paper. 

\subsection{Periodic Spaces and Unit Cells}

A $d$-periodic complex, $K \subset \R^l$ is a cell complex which permits a free action by a free abelian group $T$ of rank $d\leq l$, so that $T$ has a basis of $d$ automorphisms $\mathbf{t}_1,\dots,\mathbf{t}_d:K\to K$. 
Throughout this paper, we assume $d>0$, $K$ is locally finite so contains countably many cells, and that $T$ is a cellular action on $K$ by translations, so that the $\mathbf{t}_i$ are geometrically realised as $d$ linearly independent translations in $\R^l$.
 
A \textit{unit cell}, $U\subset K$, contains a single representative of each equivalence class of cells in $K/T$, where $K/T:=K/\sim$ for the equivalence relation $a\sim b$ if $b=\mathbf{t}(a)$ for some $\mathbf{t}\in T$.
A \textit{fundamental domain} for $T$ is a convex subset $D\subset\R^l$  such that the projection $D\to \R^l/T$ is a bijection.
It is always possible to choose $U$ so that the vertex representatives of $K/T$ belong to a fundamental domain $D$ for $T$, and higher-dimensional cell representatives have at least one vertex in $D$. We call such a $U$ a \textit{fundamental unit cell}.
Note that $U$ is in general not a subcomplex of $K$, however $\{\mathbf{t}(U)\,:\,\mathbf{t}\in T\}$ partitions $K$.
We call the smallest subcomplex of $K$ containing $U$ the \textit{closure} of $U$ and denote it by $\overline{U}$.

We say $X$ is constructed from $n_1\times\cdots\times n_d$ copies of the quotient space of $K$ \textit{with periodic boundary conditions} if $X=K/\widetilde{T}$ for the subgroup $\widetilde{T}=\langle n_1\mathbf{t}_1,\dots,n_d\mathbf{t}_d\rangle$ of $T$.
We say $Y$ is constructed from $n_1\times\cdots\times n_d$ copies of the quotient space of $K$, if $Y$ is the closure of a fundamental unit cell for $X$.

\subsection{Homology and the Fundamental Group}

For a group $G$ we denote that $N$ is a subgroup of $G$ by $N<G$ or $N\leq G$.
For a normal subgroup $N\leq G$, recall that the index of $N$ in $G$ is the cardinality of the quotient group $G/N$ and we denote this by $[G:N]$.

Given a cell complex $K$, $\pi_1(K,v)$ denotes the fundamental group of $K$ with basepoint at $v$.
If the choice of basepoint is understood or is arbitrary (up to connected component), then we will write $\pi_1(K)$ for convenience.
The commutator subgroup $[\pi_1(K), \pi_1(K)]$ is defined as the group \emph{generated} by elements of the form $[a,b] = aba^{-1}b^{-1}$.  It is straightforward to show that this subgroup is normal in $\pi_1(K)$ and the quotient group $\pi_1(K)/ [\pi_1(K), \pi_1(K)]$ is called its abelianization.

Now recall  that the chain complex of $K$, $C_\bullet(K,G)$, is the differential $\Z$-graded module where $C_k(K,G)$ is the free $G$-module whose basis is the oriented $k$-cells of $K$ and whose boundary map $\D$ sends a $k$-cell to the oriented sum of its boundary $(k-1)$-cells.
If $K$ has infinitely many $k$-cells, $C_k(K,G)$ contains only finite oriented $G$-sums of $k$-cells.
We will write $C_\bullet(K,G)=C_\bullet(K)$ whenever $G$ is understood; in this paper we typically have $G = \Z$. 
We denote by $Z_\bullet(K):=\ker(\D)$ the \textit{cycles} of $K$ and by $B_\bullet(K):=\im(\D)$ the \textit{boundaries} of $K$.
The fundamental result of homology, $\D \D = 0$, tells us that $B_{\bullet}(K)< Z_\bullet(K)$, and the \textit{cellular homology} of $K$ is then  $H_\bullet(K):=Z_\bullet(K)/B_\bullet(K)$.
The $k^\mathrm{th}$ \textit{Betti number} of $K$ (with respect to $G$) is the rank of the free part of $H_k(K)$.

In proving the results of Section~\ref{ssec:cycles} we will need the following standard result from algebraic topology. 
\begin{theorem} [ \cite{HatcherAT}, p.166 ] \label{thm:fundgrp-H1}
Given a path-connected space $X$, the degree-1 homology group is isomorphic to the abelianization of the fundamental group: $H_1(X,\Z) \simeq \pi_1(X)/[\pi_1(X),\pi_1(X) ]$. 
\end{theorem}

\subsection{Covering spaces} \label{ssec:Covering}

Given a topological space $X_0$,  a \emph{covering space} is another space $X$ together with a continuous map $\omega: X \to X_0$ such that there is an open cover $\{U_\alpha\}$ of $X_0$ with $\omega^{-1}(U_\alpha)$ a disjoint union of open sets in $X$ for each $\alpha$, and $\omega$ restricted to each of these open sets is a homeomorphism. 

Essential to our understanding of periodic cell complexes and their quotient spaces is the Galois correspondence between subgroups of $\pi_1(X_0,x_0)$ and path-connected covering spaces of $X_0$.  The following result holds when $X_0$ is path-connected, locally path-connected and semilocally simply-connected. In particular a connected, locally finite cell complex satisfies these conditions.  
\begin{theorem} [\cite{HatcherAT}, p.67 ]\label{thm:coveringspaces}
Let $X_0$ be a connected finite cell complex with fundamental group $\pi_1(X_0,x_0)$. Then the subgroups $N < \pi_1(X_0,x_0)$ are in bijection with isomorphism classes of base-point preserving path-connected covering spaces $\omega: (X ,\tilde{x}_0) \to (X_0,x_0)$, through the correspondence $N = \omega_{*} (\pi_1(X,\tilde{x}_0))$.
Moreover, if $N' < N < \pi_1(X_0,x_0)$, then the associated covering spaces satisfy $\omega': (X',\tilde{x}_0') \to (X,\tilde{x}_0)$.
In particular when $N'$ is the trivial group, the associated covering space is simply connected and called the universal covering space of $X_0$. 
\end{theorem}

A covering space $\omega: X \to X_0$ is said to be a \emph{regular} or {normal} covering space if for each $x \in X_0$ and each pair of points $\tilde{x}, \tilde{x}' \in \omega^{-1}(x)$, there is an automorphism $\alpha: X \to X$ with $ \omega \circ \alpha = \omega$,  and $\alpha(\tilde{x}) = \tilde{x}'$.  The collection of transformations preserving the covering space form a group action on $X$ called the \emph{deck transformations}, $\mathcal{D}(X)$.  When $X_0$ and $X$ are path-connected, a covering $\omega: X \to X_0$ is normal if and only if $\omega_*(\pi_1(X,\tilde{x}_0))$ is a normal subgroup of $\pi_1(X_0,x_0)$.   More generally, we have the following result. 
\begin{theorem}  [\cite{HatcherAT}, p.72 ]\label{thm:groupaction}
Suppose we have $G \subset \text{Homeo}(Y)$ with the following property: for every point $y \in Y$, there is a neighbourhood $y \in U \subset Y$ such that for $g_1, g_2 \in G$,  $g_1 (U) \cap g_2(U) \neq \emptyset$ implies $g_1 = g_2$.  Then: 
\begin{enumerate} 
 \item The quotient map $\omega: Y\to Y/G$,  $\omega(y) = Gy$, is a normal covering space.
 \item $G$ is the group of deck transformations of this covering space $Y \to Y/G$ if $Y$ is path-connected.
 \item  $G$ is isomorphic to $ \pi_1(Y /G) / \omega_*(\pi_1(Y ) )$ if $Y$ is path-connected and locally path-connected.
\end{enumerate}
\end{theorem}

\subsection{Topological Crystallography}\label{ssec:TopoCryst}

In this section we briefly summarise results about connected covering spaces of finite connected graphs developed by Sunada in his book `Topological Crystallography'~\cite{sunada2012topological}.

Consider the case that  $X_0$ is a finite connected graph with $\nu$ vertices and $\ep$ edges, we set $g = \ep - \nu +1$ (this is in fact the degree-1 Betti number of $X_0$). 
Proposition 1A.2 of \cite[p.84]{HatcherAT} establishes that $\pi_1(X_0, x_0)$ is a free group with $g$ generators.  
The universal covering space of $X_0$ is a tree, $X^{uni}$, with $\pi_1(X^{uni})$ trivial and deck transformation group $\mathcal{D}(X^{uni}) \simeq \pi_1(X_0,x_0)$.

Applying Theorem~\ref{thm:coveringspaces} to this case, we see that the commutator subgroup $[\pi_1(X_0), \pi_1(X_0)]$ is associated with a path-connected covering space $\omega^{ab}: X^{ab} \to X_0$ whose fundamental group is isomorphic to $[\pi_1(X_0), \pi_1(X_0)]$, and deck transformation group $\mathcal{D}(X^{ab}) \simeq \pi_1(X_0)/[\pi_1(X_0), \pi_1(X_0)] \simeq H_1(X_0)$, which is a \emph{free abelian group} with $g$ generators. 
Sunada calls the covering space $X^{ab}$ the \emph{maximal abelian covering graph} for $X_0$.  Any other covering space $\omega: X \to X_0$ with an abelian deck transformation group is called an \emph{abelian cover} of $X_0$. 

Recall now that for any group $G$, and normal subgroup $H < G$, the quotient $G/H$ is abelian if and only if $H$ contains the commutator subgroup of $G$.  Applying  Theorem~\ref{thm:coveringspaces} to this situation, we see that if $\omega: X \to X_0$ is any covering space such that $\mathcal{D}(X)$ is abelian, then $X^{ab}$ also covers $X$.  That is, there exists a covering map $\omega_1: X^{ab} \to X$ such that $\omega \circ \omega_1 = \omega^{ab}$.  This is the sense in which $X^{ab}$ is `maximal'. 

As Sunada shows in \cite[Ch.6.1]{sunada2012topological}, we can associate a subgroup $H < H_1(X_0,\Z)$ with an abelian cover of $X_0$, by applying the same underlying principles as in Theorems~\ref{thm:coveringspaces} and Theorem~\ref{thm:groupaction}.  In particular,  covering spaces of $X_0$ induce group homomorphisms between the deck transformation groups for $X^{uni}$, $X^{ab}$, and an arbitrary abelian cover $X$ as per the following diagram. 
\[
  \begin{tikzcd}
   \mathcal{D}(X^{uni}) \simeq  \pi_1(X_0) \arrow{r}{\varphi} \arrow[swap]{d}{h} & \mathcal{D}(X)  \\ 
      \mathcal{D}(X^{ab}) \simeq  H_1(X_0) \arrow[swap]{ur}{\mu} & 
  \end{tikzcd}
\]
In the above diagram, $\varphi : \mathcal{D}(X^{uni}) \to \mathcal{D}(X)$  is defined using the subgroup $N < \pi_1(X_0)$ associated to $X$ by Theorem~\ref{thm:coveringspaces}, and mapping the automorphism $\alpha \in \mathcal{D}(X^{uni})$ to the coset $\varphi(\alpha) = \alpha + N$.  Since $\mathcal{D}(X)$ is abelian, we know $[\pi_1(X_0), \pi_1(X_0)] < N$. 
The homomorphisms above are all surjective so we see that $\mathcal{D}(X) \simeq H_1(X_0)/ \ker{\mu}$, and $X$ is now associated to $H = \ker{\mu} < H_1(X_0)$. 
Conversely, given a subgroup $H < H_1(X_0,\Z) \simeq \Z^g$, an element, $a \in H$ corresponds to an automorphism $\alpha_a \in \mathcal{D}(X^{ab})$. The quotient space $X = X^{ab}/H$, therefore has $ \mathcal{D}(X) = H_1(X_0,\Z)/ H $.   

Lastly, the following theorem will be particularly useful in our study of periodic graphs. 
\begin{theorem} [\cite{sunada2012topological}, p.78]\label{thm:coveringhomology}
Let $\omega: X \to X_0$ be an abelian covering space for the connected finite graph $X_0$, and $\omega_*$ be the corresponding homomorphism of homology groups.  Then $\omega_*( H_1(X, \Z)) = \ker \mu < H_1(X_0,\Z)$, where $\mu$ is the map defined in the above commuting diagram. 
\end{theorem}
 
Note that in \cite{sunada2012topological}, Sunada shows how to construct geometric realisations of abelian covering graphs from a finite graph by defining a \emph{building block} for $X_0$.  This is a mapping from the edges in $X_0$ to a set of vectors in $\R^d$, akin to the weighted quotient graph we define below. Our approach is different as we start with a periodic graph $K$ and a given lattice group action on it, then study what information about $K$ can be derived from the weighted quotient graph. 
 



\section{Periodic Graphs}
\label{sec:Graphs}

In this section, $K$ denotes a $d$-periodic cellular complex with only $0$- and $1$-dimensional cells (i.e., a graph) immersed in $\R^l$.
We assume a group $T\cong \Z^d$ acts freely on $K$ by translations and preserves the cell structure of $K$. 
We adjust notation slightly from the previous section and write  $q(K) = K/T$ for the quotient space of orbits under $T$, with $q: K \to q(K)$ denoting the normal covering space.  If $K$ is path-connected then Theorem~\ref{thm:groupaction} guarantees that $T$ is the deck-transformation group for $q: K \to q(K)$.   However, we will not assume that $K$ is connected in this section unless specified.

\subsection{Weighted Quotient Graphs} \label{ssec:WQG}

We start by defining the weight of an edge in $q(K)$ as the translation offset between the end point representatives of any edge in its lift to $K$.

\begin{definition} \label{defn:WQG}
\emph{Weighted quotient graph (WQG)}.
For each edge $e \in q(K)$, we define a weight $w(e) \in T$ as follows. 
\begin{enumerate}
\item Enumerate the vertices of $q(K)$ as $\{v_1,\ldots,v_{\nu}\}$.  When $i\leq j$, any edge joining $v_i$ and $v_j$ in $q(K)$ is given the direction pointing from $v_i$ to $v_j$. 
\item For each vertex $v_i\in q(K)$ choose a fixed representative $\tilde{v}_i\in q^{-1}(v_i)$.  It is convenient, but not essential to choose these representatives in a fundamental domain for $T$. Then each vertex $a \in K$ can be written uniquely as $a = \mathbf{t}_a(\tilde{v}_i)$ for some $i$ and $\mathbf{t}_a \in T$.  
\item Each edge $(a,b) \in K$ therefore joins $\mathbf{t}_a(\tilde{v}_i)$ and $\mathbf{t}_b(\tilde{v}_j)$. Swapping the order of $a$ and $b$ if necessary we  assume $i\leq j$.
Now define $w(e) = w(q[(a,b)]) = \mathbf{t}_b - \mathbf{t}_a$.
\end{enumerate}
Observe that $w(q[(a,b)]) = \mathbf{t}_b - \mathbf{t}_a$ is independent of the choice of representative in its fiber because any other representative is a translated copy and the end points maintain the same relative offset.  
Given a basis for $T$, we can write $w(e)$ as an integer vector of coefficients in $\mathbb{Z}^d$. 
\end{definition}

WQG's are not uniquely determined by $K$, as $q(K)$ depends on the choice of $T$ (which is not necessarily maximal) and the quotient graph edge weights depend on the labelling and choice of vertex representatives $\tilde{v}_i \in K$.
However, for any WQG we can extend the domain of $w$ to directed paths of edges in $q(K)$. 
Explicitly, for a path $p$ along the directed edges $e_1,\dots,e_m$ we define $w(p)=\sum_{i=1}^m w(e_i)$ and $w(p^{-1})=-w(p)$ (where $p^{-1}$ denotes the reverse of the path $p$ in $q(K)$).
This means $w$ takes the same value on homotopy-equivalent edge-paths
and it restricts to a group homomorphism $w:\pi_1(q(K),v) \to T$ for any basepoint $v$. 

\begin{definition}\label{defn:Qwts}
Let $Q$ be a \emph{connected} WQG with weights in a given  group $T \cong \Z^d$. 
We set $W_Q$ to be the subgroup of $T$ containing weights of all cycles in $Q$,
\[
W_Q := \langle w(\ell)\,:\,\ell\text{ is a loop in }Q\rangle
\]
\end{definition}
The subgroup $W_Q$ is independent of choice of basepoint, so we see that 
\[
W_Q = \im\left(w:\pi_1(Q) \to T\right).
\]

\begin{lemma}\label{thm:zeroweightloops}
Let $K$ be a connected $d$-periodic graph with WQG $q(K) = K/T$, and weight function $w: \pi_1(q(K),v) \to T$.  
Then $\ker w = q_*( \pi_1(K,\tilde{v}) )$. 
\end{lemma}
\begin{proof}
By Theorem~\ref{thm:groupaction}, $T$ is the deck transformation group for the covering $q: K \to q(K)$.  
This means that a loop $\gamma \in \pi_1(q(K), v)$ lifts to a unique path $\tilde{\gamma}$ in $K$ that starts at $\tilde{v}$ and ends at  $\mathbf{t}_{\gamma}(\tilde{v})$.  By the construction of the weight function, we have that $\mathbf{t}_{\gamma} = w(\gamma)$.  
So $w(\gamma) = 0$ implies $\gamma$ lifts to a loop in $\pi_1(K,\tilde{v})$.  On the other hand, any loop $\tilde{\gamma} \in  \pi_1(K,\tilde{v})$ has a unique image $q(\tilde{\gamma}) \in \pi_1(q(K),v)$.  Again by construction of the weight function, we see that $w(q(\tilde{\gamma})) = 0$. 
\end{proof}

\begin{figure}[ht]
\centering
\scalebox{0.8}{
\begin{tikzpicture}
    \node[shape=circle] (v1) at (6,-1.3) {$v_1$};
    \node[shape=circle] (v2) at (7.5,1.3) {$v_2$};
    \node[shape=circle] (v3) at (9,-1.3) {$v_3$};

    \node[shape=circle,scale=0.8] (00v1) at (0,0) {$\tilde{v}_1$};
    \node[shape=circle,fill,scale=0.3] (0p1v1) at (1.2,2.2) {};
    \node[shape=circle,fill,scale=0.3] (p10v1) at (2.4,0) {};
    \node[shape=circle,fill,scale=0.3] (m1p1v1) at (-1.2,2.2) {};
    \node[shape=circle,fill,scale=0.3] (0m1v1) at (-1.2,-2.2) {};
    \node[shape=circle,fill,scale=0.3] (p1m1v1) at (1.2,-2.2) {};
    \node[shape=circle,fill=white] (p1p1v1) at (3.6,2.2) {};
    \node[shape=circle,fill=white] (p2m1v1) at (3.6,-2.2) {};
    \node[shape=circle,fill=white] (m10v1) at (-2.4,0) {};
    
    \node[shape=circle,scale=0.8] (00v2) at (0.6,1.1) {$\tilde{v}_2$};
    \node[shape=circle,fill,scale=0.3] (p10v2) at (3,1.1) {};
    \node[shape=circle,fill,scale=0.3] (m10v2) at (-1.8,1.1) {};
    \node[shape=circle,fill,scale=0.3] (0m1v2) at (-0.6,-1.1) {};
    \node[shape=circle,fill,scale=0.3] (p1m1v2) at (1.8,-1.1) {};
    \node[shape=circle,fill=white] (0m2v2) at (-1.8,-3.3) {};
    \node[shape=circle,fill=white] (p1m2v2) at (0.6,-3.3) {};
    \node[shape=circle,fill=white] (p2m2v2) at (3,-3.3) {};
    \node[shape=circle,fill=white] (m1p1v2) at (-0.6,3.3) {};
    \node[shape=circle,fill=white] (0p1v2) at (1.8,3.3) {};
    
    \node[shape=circle,scale=0.8] (00v3) at (1.2,0) {$\tilde{v}_3$};
    \node[shape=circle,fill,scale=0.3] (m10v3) at (-1.2,0) {};
    \node[shape=circle,fill,scale=0.3] (0p1v3) at (2.4,2.2) {};
    \node[shape=circle,fill,scale=0.3] (m1p1v3) at (0,2.2) {};
    \node[shape=circle,fill,scale=0.3] (0m1v3) at (0,-2.2) {};
    \node[shape=circle,fill,scale=0.3] (p1m1v3) at (2.4,-2.2) {};
    \node[shape=circle,fill=white] (m2p1v3) at (-2.4,2.2) {};
    \node[shape=circle,fill=white] (m1m1v3) at (-2.4,-2.2) {};
    \node[shape=circle,fill=white] (p10v3) at (3.6,0) {};
    
    \node[shape=circle,fill=blue,inner sep=0pt,minimum size=0.1cm] (arrowbase) at (-0.4,-0.2) {};
    \node[shape=circle,fill=white] (arrowup) at (0.8,2.0) {};
    \node[shape=circle,fill=white] (arrowright) at (2.0,-0.2) {};

    \path[every node/.style={font=\sffamily\small}]
    	(v1) edge[>=latex,->] node[left] {\scriptsize $(0,0)$} (v2)
    	(v1) edge[>=latex,bend left = 45,->] node[left] {\small $(0,-1)$} (v2)
    	(v1) edge[>=latex,->] node[above] {\scriptsize $(0,0)$} (v3)
    	(v1) edge[>=latex,bend left = -45,->] node[above] {\small $(-1,0)$} (v3)
    	(v2) edge[>=latex,->] node[right] {\scriptsize $(0,0)$} (v3)
    	(v2) edge[>=latex,bend left = 45,->] node[right] {\small $(-1,1)$} (v3)
    	
		(arrowbase) edge[>=latex,->,blue] node[below] {\small $(1,0)$} (arrowright)
		(arrowbase) edge[>=latex,->,blue] node[left] {\small $(0,1)$} (arrowup)    	
    	
    	(00v1) edge node {} (00v2)
    		edge node {} (00v3)
    		edge node {} (m10v3)
    		edge node {} (0m1v2)
    	(00v2) edge node {} (00v3)
    		edge node {} (m1p1v3)
    		edge node {} (0p1v1)
    	(00v3) edge node {} (p10v1)
    	    edge node {} (p1m1v2)
    	(0p1v1) edge node {} (0p1v2)
    		edge node {} (0p1v3)
    		edge node {} (m1p1v3)
    	(p10v2) edge node {} (p10v1)
    		edge node {} (p10v3)
    		edge node {} (0p1v3)
    		edge node {} (p1p1v1)
    	(0p1v3) edge node {} (0p1v2)
    	    edge node {} (p1p1v1)
    	(p10v1) edge node {} (p10v3)
    	    edge node {} (p1m1v2)
    	(m1p1v3) edge node {} (m1p1v2)
    	    edge node {} (m1p1v1)
    	(p1m1v3) edge node {} (p2m1v1)
    		edge node {} (p2m2v2)
    		edge node {} (p1m1v2)
    		edge node {} (p1m1v1)
    	(p1m1v1) edge node {} (p1m1v2)
    		edge node {} (p1m2v2)
    		edge node {} (0m1v3)
    	(m1p1v1) edge node {} (m1p1v2)
    		edge node {} (m2p1v3)
    		edge node {} (m10v2)
    	(m10v2) edge node {} (m2p1v3)
    		edge node {} (m10v3)
    		edge node {} (m10v1)
    	(0m1v1) edge node {} (0m1v3)
    		edge node {} (m1m1v3)
    		edge node {} (0m1v2)
    		edge node {} (0m2v2)
    	(0m1v3) edge node {} (p1m2v2)
    	    edge node {} (0m1v2)
    	(m10v3) edge node {} (0m1v2)
    	    edge node {} (m10v1);
\end{tikzpicture}
}
\caption{Left: A section of the Kagome pattern with translational basis for $T$ shown in blue. Right: A choice of weighted quotient graph with edge weights given with respect to $T$ and the given choice of vertex representatives. }
\label{fig:kagomeWQG}
\end{figure}

In Figure~\ref{fig:kagomeWQG} we illustrate a 2-periodic graph known as the Kagome structure, $K_G$ and a choice of weighted quotient graph, $q(K_G)$. 
The WQG has three vertices and six edges, so its fundamental group is generated by four loops. 
Define $e_{ij}$ and $f_{ij}$ to be the edges between $v_i$ and $v_j$ of zero and non-zero weight respectively in the weighted quotient graph $q(K_G)$.
A choice of four loops\footnote{We use the convention that $\alpha\beta$ reads ``$\alpha$ followed by $\beta$''}  that generate  $\pi_1(q(K_G))$ is then 
\[ \ell_1 = e_{12} e_{23} e_{13}^{-1}, \; \;
\ell_2 = f_{12} f_{23} f_{13}^{-1}, \;\;
\ell_3 = e_{12} f_{12}^{-1},\;\;
\ell_4 = e_{13} f_{13}^{-1} \] 
and we see that $W_q(K_G) = \langle \, (0,0), (0,0), (1,0), (0,1) \,\rangle = T $.

\subsection{Counting components} \label{ssec:components}


Connectivity of $q(K)$ does not guarantee that $K$ is connected, as the example of Fig.~\ref{fig:interwoven-lattices} shows.  
The following result was first introduced in \cite{cohen1990recognizing} and establishes that the WQG contains enough information to determine the number of path-connected components of $K$.

\begin{theorem} \label{thm:graphH0}
Suppose $q(K) = K/T$ is a WQG for a $d$-periodic graph $K$ with lattice group $T$ and that $q(K)$ has $N \geq 1$ disjoint components labelled $Q_i$.  
Then $H_0(K)$ has a basis of $\sum_{i=1}^N\,[T:W_{Q_i}]$ elements.  
Moreover, the equivalence classes of $H_0(K)$ are independent of the choice of WQG for $K$.
\end{theorem}

\begin{proof} 
Disconnected subgraphs of $q(K)$ must have disconnected lifts, so we begin by assuming $q(K)$ is connected, i.e., $N=1$.
Label the vertices of $q(K)$ by $v_1,\dots,v_\nu$ and their representatives in $K$ by $\tilde{v}_1,\dots,\tilde{v}_\nu$.
Since $q(K)$ is connected, each vertex $v_2,\dots,v_\nu$ is homologous to $v_1$.
Also, $q: K\to q(K)$ is a normal covering space, meaning any path in $q(K)$ has a unique lift up to the choice of basepoint, so each vertex of $q^{-1}(\{v_2,\dots,v_\nu\})$ is homologous to some vertex of $q^{-1}(v_1)$.
Furthermore, if there is a cycle of weight $\mathbf{t}\in T$ based at $v_1$ then for any $\mathbf{s}\in T$ this lifts to a unique path from $\mathbf{s}(\tilde{v}_1)$ to $(\mathbf{t}+\mathbf{s})(\tilde{v}_1)$.
Conversely, by translational symmetry of $K$, any path from $\tilde{v}_1$ to $\mathbf{t}(\tilde{v}_1)$ corresponds to a unique path from $\mathbf{s}(\tilde{v}_1)$ to $(\mathbf{t}+\mathbf{s})(\tilde{v}_1)$ and projects onto a cycle of weight $\mathbf{t}$ in $q(K)$.
Hence $\mathbf{t}(\tilde{v}_1)$ and $\mathbf{s}(\tilde{v}_1)$ are homologous if and only if $\mathbf{s}+W_{q(K)}=\mathbf{t}+W_{q(K)}$ as cosets.
Thus each generator of $H_0(K)$ is represented by  $\tilde{v}_1$, offset by a translation in $T/W_{q(K)}$.  

In the case that $N\geq 2$, we apply the above  to each connected component of $q(K)$ to obtain the stated result. 

Finally, the property of two vertices being homologous in $K$ is independent of the choice of $T$ and labelling of vertices in $q(K) = K/T$, and it follows that the number of homology classes in $H_0(K)$ must be the same irrespective of this choice. 
\end{proof}

\subsection{Counting cycles} \label{ssec:cycles}

We start this section by assuming that $K$ is a connected $d$-periodic graph and the group of translations $T \cong \Z^d$ acting on $K$ satisfies the conditions of Theorem~\ref{thm:groupaction}.  
Let $q: K \to  K/T = q(K)$ denote the quotient map and  $q_* :  \pi_1(K,\tilde{v}_1) \to \pi_1(q(K),v_1)$ the (injective) induced map on fundamental groups for some (arbitrary but fixed) choice of vertex $v_1 \in q(K)$ and $\tilde{v}_1 \in q^{-1}(v)$. 
Theorem~\ref{thm:groupaction} shows that $q$ is a normal covering space, that $T$ is the deck transformation group, and $T \cong \pi_1(q(K)) / q_*(\pi_1(K)) $.  
Since $T$ is abelian, it also follows that $[\pi_1(q(K)), \pi_1(q(K)) ] < q_*(\pi_1(K)) $. 
Our goal is to construct a generating set for the 1-cycles of $K$ `up to translation', although this construction will not necessarily lead to a \textit{minimal} generating set.

First, we clarify what it means for two loops in $\pi_1(K, \tilde{v})$ to be equivalent up to translation.  
Let $\gamma, \gamma' \in \pi_1(K, \tilde{v}_1)$,  $\mathbf{t} \in T$. 
Then these two loops are equivalent up to $\mathbf{t}$ if there is a path $p_{\mathbf{t}}$, of edges in $K$ from $\tilde{v}_1$ to $\mathbf{t}(\tilde{v}_1)$ such that  $\gamma' = p_{\mathbf{t}} \mathbf{t}(\gamma) p_{\mathbf{t}}^{-1}$. 
When we make the natural identification of loops in  $\pi_1(K, \tilde{v}_1)$ with cycles in $H_1(K)$, we have $[\gamma'] = [\mathbf{t}(\gamma)]$. 

The following result treats the case that $K$ is the maximal abelian covering space for $q(K)$.  

\begin{proposition}\label{thm:maxabcover}
Suppose $K$ is a connected $d$-periodic graph with translation group $T \cong \Z^d$ and quotient $q(K) = K/T$.  If $q(K)$ has $\ep$ edges and $\nu$ vertices with $\ep - \nu + 1 = d$, then $K$ is the  maximal abelian covering graph for $q(K)$.  Moreover, if $\ell_1, \ldots, \ell_d$ are generators for the fundamental group $\pi_1(q(K),v_1)$, then $\pi_1(K,\tilde{v}_1)$ is generated by lifts of paths with the form $p [ \ell_i, \ell_j ] p^{-1}$, where $p \in \pi_1(q(K),v_1)$ 
and $[ \ell_i, \ell_j ]$ is the commutator $\ell_i \ell_j \ell_i^{-1} \ell_j^{-1}$.  
\end{proposition}
\begin{proof}
As discussed in Section~\ref{ssec:Covering}, the maximal abelian covering graph has deck transformation group isomorphic to $H_1(q(K)) \cong \pi_1(q(K))/ [  \pi_1(q(K)),  \pi_1(q(K)) ] \cong \Z^d$.  
By Theorem~\ref{thm:coveringspaces}, we also have that $\pi_1(K) \cong [ \pi_1(q(K)),  \pi_1(q(K)) ] $. 
If $d = 1$, then $\pi_1(q(K)) \cong \Z$ is generated by a single loop $\ell_1$, and is abelian.  It follows that  $\pi_1(K)$ is a trivial group.  

We now assume $d\geq 2$ and consider loops in $\pi_1(q(K))$ with the form $p [ \ell_i, \ell_j ] p^{-1}$.
Note that 
\[
p [ \ell_i, \ell_j ] p^{-1} = [p,\ell_i][\ell_i,p\ell_j] \in [ \pi_1(q(K)),  \pi_1(q(K)) ].
\]
We next show that for any loops $a, b \in \pi_1(q(K))$ that $[a, b]$ takes the appropriate form by applying induction on the length of the words for $a, b$ expressed in the generators $\ell_i$. 
Clearly, when both $a, b$  have length = 1, the result holds.  So now assume that $[a, b] = z_1 z_2 \cdots z_n$ where each $z_i$ has the form $p_i [ \ell_{i_1}, \ell_{i_2}] p_i^{-1}$ for all words $a, b$ of length $n\leq m$ and consider the commutator $[ a \ell_j , b]$.  We show this is identical to the product $a [ \ell_j, b] a^{-1} [a, b]$, which has the correct form by the inductive hypothesis.  So: 
\begin{align*}
 a [ \ell_j, b] a^{-1} [a, b] & =  a \ell_j b \ell_j^{-1} b^{-1} a^{-1} a b a^{-1} b^{-1} =  a \ell_j b \ell_j^{-1} a^{-1} b^{-1} \\
 	& = a \ell_j b (a \ell_j)^{-1}  b^{-1}  = [ a \ell_j , b] . 
\end{align*}
Similarly, $[ a , b\ell_k] = [a,b] b [a, \ell_k] b^{-1}$, and we deduce that $[a', b']$ can be expressed in the desired form for all words of length at most $m+1$.  
\end{proof}

\begin{corollary}\label{thm:graphH1-maxab}
Suppose $K$ is a connected $d$-periodic graph with translation group $T \cong \Z^d$ and quotient $q(K) = K/T$.  If $q(K)$ has $\ep$ edges and $\nu$ vertices with $\ep - \nu + 1 = d \geq 2$, then the homology group $H_1(K)$ is generated by $\binom{d}{2}$ cycles `up to translation'.  If $d=1$, then $H_1(K)$ is trivial.  
\end{corollary}

\begin{proof}
The group $H_1(K)$ is the free abelian group with same generators as $\pi_1(K,\tilde{v}_1)$.  By the above Proposition~\ref{thm:maxabcover}, $\pi_1(K,\tilde{v})$ is generated by lifts of loops of the form $p [\ell_i, \ell_j] p^{-1}$ (when $d \geq 2$).  This means the loop in $K$ starts at $\tilde{v}_1$, follows the lifts of edges of $p$ to the vertex $\mathbf{t}_p(\tilde{v}_1)$, then traces the commutator loop  $[\ell_i, \ell_j]$, which returns to $\mathbf{t}_p(\tilde{v}_1)$, before tracing the edges of $p$ in reverse to finish at $\tilde{v}_1$.  In $H_1(K)$, the corresponding cycle consists of the edges in the commutator loop, all shifted by $\mathbf{t}_p$.  Since there are $d$-choose-2 commutator cycles, we have our result.   
\end{proof}

Next we consider the case that $K$ is a connected $d$-periodic graph and $q(K)$ has $\ep - \nu +1 = g $ cycles with $g > d$. 
Let $w: q(K) \to T$ denote the weight function, and its extension to paths and loops of $q(K)$. 
Theorem~\ref{thm:coveringspaces} tells us that $[ \pi_1(q(K)),\pi_1(q(K)) ] < q_*(\pi_1(K)) = \ker w$.  
But there are additional zero-weight loops in $\pi_1(q(K))$, not generated by commutators and we now explain how to count these.
 
\begin{proposition}\label{thm:abeliancover}
Suppose $K$ is a connected $d$-periodic graph with translation group $T \cong \Z^d$ and quotient $q(K) = K/T$, with   
$q(K)$ having $\ep$ edges and $\nu$ vertices with $\ep - \nu + 1 = g > d$.  
Let $w_\sharp : H_1(q(K)) \to T$ be the weight function evaluated on the homology classes of cycles.
Then $\ker w_\sharp$ is spanned by $g-d$ cycles. 
\end{proposition}

\begin{proof}
Let $\{\ell_1, \ldots, \ell_g\}$ be a minimal set of loops generating $\pi_1(q(K))$. 
$K$ is connected, so Theorem~\ref{thm:graphH0} implies that the integer coefficient span of the weight vectors $\{ w(\ell_1), \ldots, w(\ell_g) \}$ is $T$. 
The homology group $H_1(q(K))$ is the free abelian group with the same generators as $\pi_1(q(K))$ so we have that $H_1(q(K)) \cong \Z^g$. 
Denote the map sending loops to cycles by $h : \pi_1(q(K)) \to H_1(q(K))$ with $h(\gamma) = [\gamma]$. 
Then $w_\sharp $ is defined by $w_\sharp ([ \gamma ]) =  w( \gamma) $, for some choice of $\gamma \in h^{-1}([\gamma])$.  This is well-defined because if $\gamma' \in [ \gamma ]$ then the abelianisation of the words for each representative yields the same coefficients for each generator $\ell_i$. 
That is, $w(\gamma') = w(\gamma)$ if $\gamma' \gamma^{-1} \in [ \pi_1(q(K)),\pi_1(q(K)) ] $. 
It follows that the set of weight vectors $\{w_\sharp([\ell_1]), \ldots, w_\sharp([\ell_g])\}$ spans $T$ also.
Finally, observe that $w_\sharp$ is a linear transformation from $\Z^g$ onto $\Z^d$ so the dimension of $\ker w_\sharp$ is $g-d$. 
\end{proof}

\begin{theorem}\label{thm:graphH1}
Suppose $K$ is a connected $d$-periodic graph with translation group $T \cong \Z^d$ and quotient $q(K) = K/T$.  
If $q(K)$ has $\ep$ edges and $\nu$ vertices with $\ep - \nu + 1 = g > d$, then the homology group 
$H_1(K)$ is generated by $g - d +  \binom{s}{2}$ cycles up to translation, where $d \leq s \leq g$ is the number of weight vectors from a minimal generating set of loops in $q(K)$ required to span $T$. 
\end{theorem}

\begin{proof}
Let $\{ \ell_1, \ldots, \ell_g\}$ be a minimal set of generating loops for $\pi_1(q(K))$ 
and  $q^{ab}: Q^{ab} \to q(K)$ denote the maximal abelian covering space. 

Corollary~\ref{thm:graphH1-maxab} shows that $H_1(Q^{ab})$ is generated by $\binom{g}{2}$ commutator cycles $[\ell_i, \ell_j]$ lifted to $Q^{ab}$, and all their translations. 
The fact that $[\pi_1(q(K)), \pi_1(q(K))]   < q_*(\pi_1(K))$ means that 
$H_1(K)$ also contains lifts of the $\binom{g}{2}$ commutator cycles $[\ell_i, \ell_j]$.  

Consider the mapping $q_\sharp: H_1(K) \to H_1(q(K))$ induced by the covering map $q$.  
Theorem~\ref{thm:coveringhomology} states that the image $q_\sharp(H_1(K)) = \ker w_\sharp$, which has dimension $g-d$ by Proposition~\ref{thm:abeliancover}.  
Observe that $q_\sharp$ maps every commutator cycle in $H_1(K)$  to the trivial cycle in $H_1(q(K))$, so the non-trivial elements of $\ker w_\sharp$ lift to cycles in $H_1(K)$ that are not generated by commutators. 

The final step is to establish $s$.  Let $L = \{ \ell_1, \ldots, \ell_g\}$ be a set of loops that generate $\pi_1(q(K),v)$, and let $S_L = \{w(\ell_1), \ldots, w(\ell_g)\}$ be the corresponding set of weight vectors.  These must be a spanning set for $T$, so $d \leq s \leq g$.
Note $s$ may be greater than $d$, since $\Z$ is not a vector space and therefore $S_L$ may not have a subset which is a basis of $T$.
It is nevertheless possible that a subset of $S_L$ spans $T$ so we set 
\[  s =   \min_L \min \{ |S| \,:\, S\subseteq S_L \text{ and $T = \Z$-span of $S$} \} \]
It follows that the commutator cycles in $H_1(K)$ are generated by the $\binom{s}{2}$ pairs of loops in the minimal spanning set $S$. 
\end{proof}

Next, we consider the case that $K$ is not connected but $q(K)$ \textit{is} connected, as occurs in Figure~\ref{fig:interwoven-lattices}.
The existence of multiple components of $K$ means the fundamental group represents cycles only in a single component, while the homology group encodes cycles in all components of $K$. 
Explicitly, $\gamma\in\pi_1(K,\tilde{v})$ and $\tr(\gamma)\in\pi_1(K,\tr(\tilde{v}))$ with $\tr\not\in W_{q(K)}$ can never be equivalent up to translation as loops, but after making the natural identification with cycles we nonetheless recover $[\tr(\gamma)]=\tr[\gamma]$.
Since $\tr$ generates a homeomorphism between components, we may weaken the requirements of Theorem~\ref{thm:graphH1} to allow that $K$ is disconnected but $q(K)$ is connected.

Further, if $q(K)$ is disconnected, we take the sum of contributions from each connected component and sum them together.
If $q(K)$ has $N$ connected components, $Q_i$ with $g_i = \ep_i \nu_i +1$, this means $H_1(K)$ is generated by $ \sum_{i=1}^N \left(g_i - d_i + \binom{s_i}{2}\right)$ cycles up to translation, where $T_i \leq T$ is a sublattice of dimension $d_i$, spanned by edge-weights in $Q_i$,   $d_i \leq s_i \leq g_i$ is the number of weight vectors from a minimal generating set of loops in the $i^\mathrm{th}$ component of $q(K)$ required to span $T_i$.

Finally, we remark that \textit{counting} cycles or finding a basis of $H_1(K)$ (even up to translation) is complicated by various linear dependencies.
For example, in the 1-skeleton of the $3$-dimensional cube $C_3$, it is well known that five square faces form a basis for $H_1(C_3)$ and the sixth square face can be written as a sum of the remaining five.
More generally, similar dependencies will exist for most periodic graphs (e.g., anything not \emph{embedded} in $\R^2$). 
For this reason the generating set for $H_1(K)$ in Theorem~\ref{thm:graphH1} is not minimal, and $g-d + \binom{s}{2}$ is  an upper bound on the Betti number ``per unit cell'' for $K$. 
Indeed, we have the following estimate for $\beta_1$ of particular subsets of $K$. 

\begin{proposition}\label{thm:graphH1betti}
With the same notation as in Theorem~\ref{thm:graphH1}, suppose $K$ is a connected $d$-periodic graph in $\R^l$, and let $Y_n \subset K$ be constructed from $n^d$ copies of a fundamental unit cell for $K$.  Then the first Betti number for $\overline{Y}_n$ satisfies 
\[   \lim_{n \to \infty} \frac{ \beta_1( \overline{Y}_n) }{ n^d }  =   \ep - \nu.   \] 
\end{proposition}
\begin{proof}
Recall that a fundamental unit cell $U$ has $\nu$ vertices chosen in a fundamental domain $D \subset \R^l$ and $\ep$ edges, with each edge having at least one vertex in $D$. 
The subset $Y_n$ contains $n^d$ copies of $U$, arranged in a regular $n \times \ldots n $ block, so $Y_n$ has $\nu n^d$ vertices and $\ep n^d$ edges. The closure  $\overline{Y}_n$  differs from $Y_n$ by adding extra vertices to edges that have only one end point in $Y_n$.  There are at most $M d \ep n^{d-1}$ such edges, where $M$ is the magnitude of the largest coefficient in any of the weight vectors of $q(K)$. 
Now we see that the Euler characteristic,  \[ \nu n^d - \ep n^d \leq \chi ( \overline{Y}_n) \leq \nu n^d  + M d \ep n^{d-1} - \ep n^d . \]

Our next task is to bound $\beta_0(\overline{Y}_n)$. $K$ is connected, but $\beta_0(\overline{Y}_n)$ may have multiple components, even for arbitrarily large $n$. 
So consider the set of vertex representatives 
$\{\tilde{v}_1,\ldots,\tilde{v}_\nu \} \subset U$ and the 1-chains $p_{ij}$ that make $\tilde{v}_i$ homologous to $\tilde{v}_j$. For any 1-chain $p_{ij}$, there will be a vertex ``furthest from $U$'' in the sense that this vertex $a = \mathbf{t}_a(\tilde{v}_k)$ for some $\tilde{v}_k \in U$, and $\mathbf{t}_a$ has the largest coefficient $c_l$ with respect to the basis chosen for $T$.  For each pair $\tilde{v}_i, \tilde{v}_j$ define
\[ M_{ij} = \min_{p_{ij}} \max_{e \in p_{ij}} 
\max_{l} \left\{ | c_l | \; : \; e = (a,b), \; a = \mathbf{t}_a(\tilde{v}_k) \text{ and } \mathbf{t}_a = \sum c_l \mathbf{t_l}  \right\},  \]
and let $M_0 = \max \{ M_{ij} \text{ for } i,j \in \{1, \ldots, \nu\} \}$.
Then for large $n$, consider $Y_{n-2M_0}$ as a centered subset of $Y_n$; the vertices of $ Y_{n-2M_0}$ must be connected by 1-chains contained in $Y_n$. The number of vertices in $Y_n \setminus Y_{n-2M_0}$ is then an upper bound on the number of components of $\overline{Y}_n$.  It follows that $1 \leq \beta_0 (\overline{Y}_n) \leq  2 d \nu M_0 n^{d-1}$. 
Finally, we have $\beta_1 =  \beta_0 - \chi $ so 
\[ \beta_0(\overline{Y}_n) - (\nu - \ep) n^d \geq \beta_1(\overline{Y}_n) \geq  \beta_0(\overline{Y}_n) - M d \ep n^{d-1} - (\nu - \ep) n^d \] 
and the result follows.  
\end{proof}

\subsection{Examples}

We now illustrate the above results with some worked examples. 

\subsubsection{Simple cubic} 
We start with the simple cubic graph $K_S \subset \R^3$ which has vertices at points with integer coordinates, and an edge joining $v_1, v_2$ if and only if $|| v_2 - v_1 ||_1 = 1$. The maximal translational symmetry group $T \cong \Z^3$. 
The weighted quotient graph $q(K_S)$ has a single vertex and three loop edges, $\ell_1, \ell_2, \ell_3$, with weights $w(\ell_1) = (1,0,0)$, $w(\ell_2) = (0,1,0)$ and $w(\ell_3) = (0,0,1)$. These clearly span $T$, and since $\ep - \nu + 1 = 3 - 1 + 1 = 3 = d$, we see that $K_S$ is the maximal abelian cover for $q(K_S)$ with the homology group $H_1(K_S)$ generated by the three commutator cycles and all their translations. 
\[  \gamma_{12} = \tilde{\ell_1} + t_x(\tilde{\ell_2}) - t_y(\tilde{\ell_1}) - \tilde{\ell_2}, \quad
\gamma_{23} = \tilde{\ell_2} + t_y(\tilde{\ell_3}) - t_z(\tilde{\ell_2}) - \tilde{\ell_3}, \quad
\gamma_{31} = \tilde{\ell_3} + t_z(\tilde{\ell_1}) - t_x(\tilde{\ell_3}) - \tilde{\ell_1}. 
\]
On the other hand, using the Euler characteristic calculation as in the proof of Proposition~\ref{thm:graphH1betti} the number of independent cycles in an $n^3$ subset of $K_S$ is $\beta_1(n) = 2 n^3 - 3 n^2 + 1$. 

\subsubsection{Interwoven cubic} 
Now consider the example of Figure~\ref{fig:interwoven-lattices}, $K_I$, consisting of two copies of the simple cubic graph, one copy with vertices at integer coordinates and the other with vertices at half-integer coordinates. The maximal translation symmetry for this structure is the group $T_I$ generated by $t_1 = (-\frac12,\frac12,\frac12)$, $t_2 = (\frac12,-\frac12,\frac12)$ and $t_3 = (\frac12,\frac12,-\frac12)$.
The quotient graph is again the $3$-bouquet graph, with weights defined by coordinates with respect to $\{t_1, t_2, t_3\}$ so that $w(\ell_1) = (0,1,1)$, $w(\ell_2) = (1,0,1)$ and $w(\ell_3) = (1,1,0)$.   The subgroup $W_Q$ has index-2 in $T_I$, reflecting the existence of the two cosets in $H_0(K_I)$, namely $[\tilde{v}_1 ] = [(0,0,0)]$ and $[t_1(\tilde{v}_1)] = [(-1/2, 1/2,1/2)]$. 

We count the cycles in $K_I$ up to translation by noting that $\ker w_\sharp$ is trivial and so $H_1(K_I)$ is generated by the three commutator cycles, $\gamma_{12}, \gamma_{23}, \gamma_{31}$ with their translations. Note that these commutators have the same expression in $\pi_1(q(K_I),\tilde{v}_1)$ as they did in $\pi_1(q(K_S))$, but in their lifts to $H_1(K_I)$, we need to replace $t_x$ by $w(\ell_1) = t_2 + t_3$, and so on.

\subsubsection{Triangular grid}
A planar triangular grid, $K_P$, and associated weighted quotient graph are illustrated in Figure~\ref{fig:four-cases}. 
In this example we have $T \cong \Z^2$, and weights for the three loops of $q(K_P)$ are  $w(\ell_1) = (1,0)$, $w(\ell_2) = (0,1)$ and $w(\ell_3) = (1,1)$, with respect to the standard basis. Now we have $g = 3-1+1 = 2 > d$ and must study $\ker w_\sharp$ to find loops that are not generated by commutators. 
Clearly the cycle $\gamma = [\ell_1 + \ell_2 - \ell_3] \in H_1(q(K_P))$ which has $w_\sharp (\gamma) = 0$ is an appropriate choice. 
Finally, we study the set $S = \{ (1,0), (0,1), (1,1)\}$ and see that the first two weight-vectors span $T$ so that $s = 2$ and only one of the three commutator cycles is required as a generator in $H_1(K_P)$. In summary, we can choose the following two cycles as generators up to translation for $H_1(K_P)$: 
\[ \gamma = \tilde{\ell_1} + t_x(\tilde{\ell_2}) - \tilde{\ell_3}  , \quad 
\gamma_{12} = \tilde{\ell_1} + t_x(\tilde{\ell_2}) - t_y(\tilde{\ell_1}) - \tilde{\ell_2} .
\]
The Euler characteristic calculation yields 
$\chi(n) = -2n^2 + 2n + 1$ and $\beta_1(n) = 2n^2 - 2n$.

\subsubsection{Another abelian cover of the 3-bouquet graph}

\begin{figure}[ht]
\centering
\begin{tikzpicture}
    \clip (-0.5,-0.5) rectangle (11,4.5);
    \foreach \i in {0,...,4}
        {
        \foreach \j in {0,...,4}
            {
            \node[shape=circle,fill=black,inner sep=0pt,minimum size=0.15cm] at (\i,\j) { };
            \draw[line width=0.2mm] (\i-0.5,\j-0.5) to (\i+0.5,\j+0.5);
            }
        \foreach \j in {-3,...,4}
            {
            \draw[line width=0.2mm] (\i,\j) to [bend left = 25] (\i,\j+3);
            }
        }

    \foreach \i in {-2,...,4}
        {
        \foreach \j in {0,...,4}
            {
            \draw[line width=0.2mm] (\i,\j) to [bend right = 25] (\i+2,\j);
            }
        }

    \fill[white] (4.5,-0.5) rectangle (7,4.5);
    
    \node[shape=circle,fill=black,inner sep=0pt,minimum size=0.15cm] (quotient) at (8.5,2) { };
    
    \node[shape=circle,fill=white,scale=0.9] () at (9.8,0.7) {$(1,1)$};
    \node[shape=circle,fill=white,scale=0.9] () at (8.5,4) {$(2,0)$};
    \node[shape=circle,fill=white,scale=0.9] () at (7.2,0.7) {$(0,3)$};
    \draw[>=latex,->,line width=0.3mm] (quotient) to [out=360,in=300,looseness=80] (quotient);
	\draw[>=latex,->,line width=0.3mm] (quotient) to [out=120,in=60,looseness=80] (quotient);
	\draw[>=latex,->,line width=0.3mm] (quotient) to [out=240,in=180,looseness=80] (quotient);

\end{tikzpicture}
\caption{A section of the 2-periodic graph $K_A$ (left) and its WQG with edge weights $(1,1)$, $(2,0)$ and $(0,3)$ (right). The vertices of $K_A$ are at all points with integer coordinates and the translation group $T \cong \Z^2$ has the standard basis.}
\label{fig:new-example}
\end{figure}

The periodic graph $K_A$ in Figure~\ref{fig:new-example} demonstrates the case when the set $S$ of weight vectors constructed in the proof of Theorem~\ref{thm:graphH1}, requires $s > d$ elements to span $T$.  
For ease of notation, let $\ell_1$ be the edge of $q(K_A)$ of weight $(1,1)$, let $\ell_2$ be the edge of weight $(2,0)$ and $\ell_3$ the edge of weight $(0,3)$.
First we look for a cycle that generates $\ker w_\sharp$, and see by inspection that $\gamma = [ 6 \ell_1 - 3\ell_2 - 2\ell_3] \in H_1(q(K_A))$ is suitable. 
Studying $S = \{(1,1), (2,0), (0,3)\}$, we quickly see that the integer span of each pair of weight vectors generates a strict subgroup of $T$. This is why $s = 3 > 2$. It follows that appropriate lifts of all three commutator loops $\gamma_{ij}$  are required to generate $H_1(K_A)$ up to translation. 

The Euler characteristic computation is a little more involved than the previous examples but we find 
$\chi(n) = -2 n^2 + 3n + 2n + 1$ and 
$\beta_1(n) = 2n^2 - 5n$ for $n \geq 3$.

\subsubsection{Kagome revisited}

We saw at the end of Section~\ref{ssec:WQG}, that the Kagome pattern has a WQG with three vertices and six edges, Figure~\ref{fig:kagomeWQG}.  
The four generating cycles we chose for $\pi_1(q(K_G))$ have weight vectors $ S = \{ (0,0), (0,0), (1,0), (0,1) \}$. 
It follows immediately that the cycles $[\ell_1]$ and $[\ell_2]$ are in $\ker w_\sharp$, and that the weight-vectors for $\ell_3$ and $\ell_4$ generate $T \cong \Z^2$. 
We therefore have the following choice of generating cycles for $H_1(K_G)$: 
\begin{align*}
\gamma_1 & = [\tilde{\ell}_1] = [\tilde{e}_{12} + \tilde{e}_{23} - \tilde{e}_{13}], \\
\gamma_2 & = [\tilde{\ell}_2] = [\tilde{f}_{12} + t_{(0,-1)}(\tilde{f}_{23}) - \tilde{f}_{13}], \\
\gamma_{34} & = [\tilde{\ell_3} + t_{(0,1)}(\tilde{\ell_4}) - t_{(1,0)}(\tilde{\ell_3}) - \tilde{\ell_4}] \\
 & = [ \tilde{e}_{12} - t_{(1,0)}(\tilde{f}_{12}) + t_{(1,0)}(\tilde{e}_{13}) - t_{(1,1)}(\tilde{f}_{13}) +t_{(1,1)}(\tilde{f}_{12}) - t_{(1,0)}(\tilde{e}_{12}) +t_{(1,0)}(\tilde{f}_{13}) - \tilde{e}_{13} ] .
\end{align*}
Finally, the Euler characteristic for $n\times n$ unit cells of the Kagome pattern is $\chi(n) = (3-6)n^2 + 2n + n$, and so $\beta_1(n) = 3n^2 - 3n + 1$.

\begin{figure}[ht]
\centering
\scalebox{0.85}{
\begin{tikzpicture}
    \node[shape=circle] (v1) at (6,-1.3) {$v_1$};
    \node[shape=circle] (v2) at (7.5,1.3) {$v_2$};
    \node[shape=circle] (v3) at (9,-1.3) {$v_3$};

    \node[shape=circle,scale=0.8] (00v1) at (0,0) {$\tilde{v}_1$};
    \node[shape=circle,fill,scale=0.3] (0p1v1) at (1.2,2.2) {};
    \node[shape=circle,fill,scale=0.3] (p10v1) at (2.4,0) {};
    \node[shape=circle,fill,scale=0.3] (m1p1v1) at (-1.2,2.2) {};
    \node[shape=circle,fill,scale=0.3] (0m1v1) at (-1.2,-2.2) {};
    \node[shape=circle,fill,scale=0.3] (p1m1v1) at (1.2,-2.2) {};
    \node[shape=circle,fill=white] (p1p1v1) at (3.6,2.2) {};
    \node[shape=circle,fill=white] (p2m1v1) at (3.6,-2.2) {};
    \node[shape=circle,fill=white] (m10v1) at (-2.4,0) {};
    
    \node[shape=circle,scale=0.8] (00v2) at (0.6,1.1) {$\tilde{v}_2$};
    \node[shape=circle,fill,scale=0.3] (p10v2) at (3,1.1) {};
    \node[shape=circle,fill,scale=0.3] (m10v2) at (-1.8,1.1) {};
    \node[shape=circle,fill,scale=0.3] (0m1v2) at (-0.6,-1.1) {};
    \node[shape=circle,fill,scale=0.3] (p1m1v2) at (1.8,-1.1) {};
    \node[shape=circle,fill=white] (0m2v2) at (-1.8,-3.3) {};
    \node[shape=circle,fill=white] (p1m2v2) at (0.6,-3.3) {};
    \node[shape=circle,fill=white] (p2m2v2) at (3,-3.3) {};
    \node[shape=circle,fill=white] (m1p1v2) at (-0.6,3.3) {};
    \node[shape=circle,fill=white] (0p1v2) at (1.8,3.3) {};
    
    \node[shape=circle,scale=0.8] (00v3) at (1.2,0) {$\tilde{v}_3$};
    \node[shape=circle,fill,scale=0.3] (m10v3) at (-1.2,0) {};
    \node[shape=circle,fill,scale=0.3] (0p1v3) at (2.4,2.2) {};
    \node[shape=circle,fill,scale=0.3] (m1p1v3) at (0,2.2) {};
    \node[shape=circle,fill,scale=0.3] (0m1v3) at (0,-2.2) {};
    \node[shape=circle,fill,scale=0.3] (p1m1v3) at (2.4,-2.2) {};
    \node[shape=circle,fill=white] (m2p1v3) at (-2.4,2.2) {};
    \node[shape=circle,fill=white] (m1m1v3) at (-2.4,-2.2) {};
    \node[shape=circle,fill=white] (p10v3) at (3.6,0) {};
    
    \node[shape=circle,fill=white] (arrowup) at (0.8,2.0) {};
    \node[shape=circle,fill=white] (arrowright) at (2.0,-0.2) {};

    \path[every node/.style={font=\sffamily\small}]
    	(v1) edge[>=latex,->] node[left] {\small $e_{12}$} (v2)
    	(v1) edge[>=latex,bend left = 45,->] node[left] {\small $f_{12}$} (v2)
    	(v1) edge[>=latex,->] node[above] {\small $e_{13}$} (v3)
    	(v1) edge[>=latex,bend left = -45,->] node[above] {\small $f_{13}$} (v3)
    	(v2) edge[>=latex,->] node[right] {\small $e_{23}$} (v3)
    	(v2) edge[>=latex,bend left = 45,->] node[right] {\small $f_{23}$} (v3)
    	
    	
    	(00v1) edge[blue, thick] node[left] {\small $\tilde{e}_{12}$} (00v2)
    		edge[blue, thick] node[below] {\small $\tilde{e}_{13}$} (00v3)
    		edge[blue, thick] node[above] {\small $\tilde{f}_{13}$} (m10v3)
    		edge[blue, thick] node[right] {\small $\tilde{f}_{12}$} (0m1v2)
    	(00v2) edge[blue, thick] node[right] {\small $\tilde{e}_{23}$} (00v3)
    		edge[blue, thick] node[left] {\small $\tilde{f}_{23}$} (m1p1v3)
    		edge node {} (0p1v1)
    	(00v3) edge node {} (p10v1)
    	    edge node {} (p1m1v2)
    	(0p1v1) edge node {} (0p1v2)
    		edge node {} (0p1v3)
    		edge node {} (m1p1v3)
    	(p10v2) edge node {} (p10v1)
    		edge node {} (p10v3)
    		edge node {} (0p1v3)
    		edge node {} (p1p1v1)
    	(0p1v3) edge node {} (0p1v2)
    	    edge node {} (p1p1v1)
    	(p10v1) edge node {} (p10v3)
    	    edge node {} (p1m1v2)
    	(m1p1v3) edge node {} (m1p1v2)
    	    edge node {} (m1p1v1)
    	(p1m1v3) edge node {} (p2m1v1)
    		edge node {} (p2m2v2)
    		edge node {} (p1m1v2)
    		edge node {} (p1m1v1)
    	(p1m1v1) edge node {} (p1m1v2)
    		edge node {} (p1m2v2)
    		edge node {} (0m1v3)
    	(m1p1v1) edge node {} (m1p1v2)
    		edge node {} (m2p1v3)
    		edge node {} (m10v2)
    	(m10v2) edge node {} (m2p1v3)
    		edge node {} (m10v3)
    		edge node {} (m10v1)
    	(0m1v1) edge node {} (0m1v3)
    		edge node {} (m1m1v3)
    		edge node {} (0m1v2)
    		edge node {} (0m2v2)
    	(0m1v3) edge node {} (p1m2v2)
    	    edge node {} (0m1v2)
    	(m10v3) edge node {} (0m1v2)
    	    edge node {} (m10v1);
\end{tikzpicture}
}
\caption{ A section of the Kagome pattern $K_G$ and its WQG $q(K_G)$ with edge labels for a choice  representative in the cover. Refer to \protect{Figure~\ref{fig:kagomeWQG}} for the edge weights.  }
\label{fig:kagome-cycles}
\end{figure}


\section{Periodic Cellular Complexes}
\label{sec:Complexes}

In this section, $K$ denotes a $d$-periodic cellular complex equipped with a group of translations $T\cong \Z^d$ acting on $K$, and $q:K\to q(K):=K/T$ denotes the canonical quotient map.
We fix an orientation for each cell $q(K)$, which induces an orientation of all cells in $K$.
Further, chain complexes and homology groups are now understood to have coefficients in some field $\F$ so that homology groups are vector spaces.

\subsection{Quotient Spaces}

For cell complexes with dimension $k\geq 2$ it is difficult to generalise the notion of a weighted quotient graph to a ``weighted quotient space''.
Weights on edges encode the relative offset of boundary vertices, allowing us to distinguish between $1$-cycles of $q(K)$ that lift to true cycles in $K$ and those which are essentially a path through the periodic structure.
For higher dimensions there is no such canonical pairing of $\partial \sigma$ with a weight on $\sigma$.
Without this information there are several cases we cannot in general uncouple when looking at the quotient space.

\begin{lemma} \label{lem:four-cases}
Suppose $\gamma\in C_\bullet(K)$ is such that $q(\gamma)\in Z_\bullet(q(K))$. Then exactly one of the following holds
\begin{enumerate}
\item $\gamma\in Z_\bullet(K)$ and $q(\gamma)=0$
\item $\gamma\in Z_\bullet(K)$ and $q(\gamma)\neq 0$
\item $\gamma\not\in Z_\bullet(K)$ and $q(\gamma)=0$
\item $\gamma\not\in Z_\bullet(K)$ and $q(\gamma)\neq 0$
\end{enumerate}
\end{lemma}
In Case~3 and Case~4 this implies that $\partial(\gamma)$ is a non-zero element of $\ker(q)$, which in Case~4 is how $q(\gamma)$ passes to a non-trivial cycle of $q(K)$.
See Figure~\ref{fig:four-cases} for an illustration of each case for a periodic graph.
\begin{proof}
The quotient map $q$ induces a surjective homomorphism  $q_\bullet:C_\bullet(K)\to C_\bullet(q(K))$ which we also denote by $q$.
The four cases above are mutually exclusive, so the result follows by the assumption that $0 = \partial q(\gamma)=q\left(\partial\gamma\right)$.   Case~3 (resp.~4) occurs when $q(\gamma)=0$ ($q(\gamma)\neq 0$) and Case~1 (Case~2) fails.
\end{proof}

Note that in Case~1, $\gamma$ is a cycle in $C_\bullet(K)$ which disappears in $C_\bullet(q(K))$ while in Case~4 we have gained a cycle in $C_\bullet(q(K))$ from a non-cycle in $C_\bullet(K)$.  
This captures the basic obstruction to determining the homology of $K$ from the homology of $q(K)$, without the additional information provided by the weights.

\begin{figure}
\centering
\begin{tikzpicture}
    \foreach \i in {-1,...,4}
    {
    \foreach \j in {-1,...,4}
    {
    \ifthenelse{\i=-1 \OR \i=4 \OR \j=-1 \OR \j=4}{\node[shape=circle,fill=white] (\i,\j) at (\i,\j) { };}{\node[shape=circle,fill=black,inner sep=0pt,minimum size=0.15cm] (\i,\j) at (\i,\j) { };}
    }
    }

    \node[shape=circle,fill=black,inner sep=0pt,minimum size=0.15cm] (quotient) at (7.5,1.5) { };
    
    \node[shape=circle,fill=white,scale=0.9] () at (8.8,0.2) {$(1,0)$};
    \node[shape=circle,fill=white,scale=0.9] () at (7.5,3.5) {$(0,1)$};
    \node[shape=circle,fill=white,scale=0.9] () at (6.2,0.2) {$(1,1)$};
    \draw[>=latex,->,line width=0.3mm] (quotient) to [out=360,in=300,looseness=80] (quotient);
	\draw[>=latex,->,line width=0.3mm] (quotient) to [out=120,in=60,looseness=80] (quotient);
	\draw[>=latex,->,line width=0.3mm] (quotient) to [out=240,in=180,looseness=80] (quotient);	
    
    \path[every node/.style={font=\sffamily\small}]
    	
    	(0,0) edge[>=latex,red,line width=0.4mm,->] node {} (1,0)
    		edge[>=latex,red,line width=0.4mm,<-] node {} (0,1)
    		edge node {} (1,1)
    		edge node {} (-1,0)
    		edge node {} (-1,-1)
    		edge node {} (0,-1)
    	(0,1) edge[>=latex,red,line width=0.4mm,<-] node {} (1,1)
    		edge node {} (-1,1)
    		edge node {} (0,2)
    		edge node {} (1,2)
    		edge node {} (-1,0)
    	(1,0) edge[>=latex,red,line width=0.4mm,->] node {} (1,1)
    		edge node {} (1,-1)
    		edge node {} (2,1)
    		edge node {} (2,0)
    		edge node {} (0,-1)
    	(2,0) edge node {} (2,1)
    		edge node {} (2,-1)
    		edge node {} (3,1)
    		edge node {} (1,-1)
    		edge[>=latex,green!50!gray,line width=0.4mm,->] node {} (3,0)
    	(3,0) edge node {} (3,1)
    		edge node {} (3,-1)
    		edge node {} (4,1)
    		edge node {} (4,0)
    		edge node {} (2,-1)
    	(1,1) edge node {} (1,2)
    		edge node {} (2,2)
    		edge node {} (2,1)
    	(2,1) edge node {} (2,2)
    		edge node {} (3,2)
    		edge[>=latex,green!50!gray,line width=0.4mm,<-] node {} (3,1)
    	(3,1) edge node {} (3,2)
    		edge node {} (4,2)
    		edge node {} (4,1)
    	(0,2) edge node {} (-1,2)
    		edge node {} (-1,1)
    		edge[>=latex,orange,line width=0.4mm,<-] node {} (0,3)
    		edge[>=latex,orange,line width=0.4mm,->] node {} (1,3)
    		edge node {} (1,2)
    	(0,3) edge node {} (-1,2)
    		edge node {} (-1,3)
    		edge node {} (0,4)
    		edge node {} (1,4)
    		edge[>=latex,orange,line width=0.4mm,<-] node {} (1,3)
    	(1,2) edge node {} (1,3)
    		edge node {} (2,2)
    		edge node {} (2,3)
    	(1,3) edge node {} (1,4)
    		edge node {} (2,4)
    		edge node {} (2,3)
    	(2,2) edge node {} (3,2)
    		edge[>=latex,blue,line width=0.4mm,->] node {} (3,3)
    		edge node {} (2,3)
    	(2,3) edge node {} (2,4)
    		edge node {} (3,3)
    		edge node {} (3,4)
    	(3,2) edge node {} (4,2)
    		edge node {} (4,3)
    		edge node {} (3,3)
    	(3,3) edge node {} (4,3)
    		edge node {} (4,4)
    		edge node {} (3,4);
\end{tikzpicture}
\caption{A section of a $2$-periodic graph and its weighted quotient graph with respect to translations by $\Z^2$.
In red is a cycle satisfying Case~1 of Lemma~\ref{lem:four-cases}, in orange a cycle satisfying Case~2, in green a chain satisfying Case~3, and in blue a chain satisfying Case~4.}
\label{fig:four-cases}
\end{figure}

Since the boundary map is trivial in degree-$0$, $\ker(q_0)\leq Z_0(K)=C_0(K)$, the $0$-cycles of periodic cell-complexes 
all satisfy Cases~1 and~2 of Lemma~\ref{lem:four-cases}.
In the case of a periodic graph, Theorem~\ref{thm:graphH0} essentially determines when a $0$-cycle satisfying Case~1 is also a boundary.
Theorem~\ref{thm:graphH1} also helps us to distinguish between the four cases in degree-$1$ homology.
A $1$-cycle in $Z_1(q(K))$ of zero weight exactly corresponds to Case~2, those of non-zero weight exactly correspond to Case~4.
There are no $2$-simplices in a periodic graph, so this entirely defines the degree-$1$ homology.
For higher dimension, without a well-defined ``weighted quotient space'', we have no analogous tool to recover the homology of $K$ from $q(K)$ and distinguish between cycles in $q(K)$ satisfying Case~2 and Case~4, nor to recover cycles in $K$ satisfying Case~1.

\subsection{Finite Approximations}

An alternative way to calculate the homology of $K$ is to approximate it from finite subcomplexes of increasing size. 
Recall the definitions of the complexes $X$ and $Y$ in Section~\ref{sec:Definitions}, built using a sub-lattice of $T$.
Consider the simplified case that $\tilde{T}_n = \langle n\mathbf{t}_1, \ldots, n\mathbf{t}_d \rangle$ so that $\tilde{T}_n$ is an index-$n^d$ sub-lattice of $T$ and write 
$X_n = K/\tilde{T}_n$ and $Y_n$ for the closure of a fundamental unit cell for $X_n$.
It is possible to choose the $Y_n$ so that for each $m\leq n$ we have a natural embedding $i: Y_m\inc Y_n$.
This creates a directed system $(Y_n,i)$ over $\N$ whose direct limit is $K$.
In the case that $m$ divides $n$, we have that $\tilde{T}_n < \tilde{T}_m$, so that $X_n$ covers $X_m$ and there is a covering map $p: X_n \surj X_m$. By Theorem~\ref{thm:coveringspaces} we also have that $K$ covers $X_n$ for all $n$.
Applying the homology functor to both systems we have
\begin{align*}
H_\bullet(Y_1)\xrightarrow{i}H_\bullet(Y_2)\xrightarrow{i}\cdots\xrightarrow{i}H_\bullet(Y_n)\xrightarrow{i}\cdots \phantom{\cdots \xrightarrow{i}\cdots\xrightarrow{p}}&\phantom{\cdots} \\ 
\cdots \xrightarrow{i}H_\bullet(K)\xrightarrow{p}&\cdots \\
& \cdots\xrightarrow{p}H_\bullet(X_n)\xrightarrow{p}\cdots \xrightarrow{p}H_\bullet(X_2)\xrightarrow{p}H_\bullet(X_1).
\end{align*}
In the right side of the diagram, we have slightly abused notation in writing all homologies in a single line, as $H_\bullet(X_n)\xrightarrow{p}H_\bullet(X_m)$ exists if and only if $m|n$.
The diagram shows it is reasonable to approximate the homology of $K$ with $X_n$ or $Y_n$ for any $n\in\N$.
However, as remarked at the end of  Section~\ref{ssec:cycles} there are some subtleties we must be careful of with this approach.

For $\overline{U}$ the closure of a fundamental unit cell of $K$, we can write every $k$-cell of $X_n$ (not necessarily uniquely) as $\mathbf{t}(\sigma)$ where $\sigma\in \overline{U}$ and $\mathbf{t}\in (\Z/n\Z)^d$.
Thus $\dim C_k(X_n)\leq n^d\cdot\dim C_k(\overline{U})$ and for large $n$ we can thus bound
\[
\beta_k(X_n) \leq \dim C_k(X_n) \leq n^d\cdot\dim C_k(\overline{U}) = O(n^d)
\]
That is, the homology of $X_n$ is bounded by polynomial growth.
We could also obtain a similar result for $Y_n$.
Not all behaviour will be regular with respect to $n$. 
For example, when $K$ is the interwoven cubical lattices in Figure~\ref{fig:interwoven-lattices} then for $n$ odd $\beta_0(X_n)=1$ but for $n$ even $\beta_0(X_n)=2$.
The difficulty in using finite approximations for $K$ thereby lies with balancing the following. 
\begin{itemize}
    \item When is $n$ sufficiently large that $X_n$ or $Y_n$ contain a set of cycles and boundaries which generate the homology of $K$ up to translation?
    \item When is $H_\bullet(X_n)$ or $H_\bullet(Y_n)$ too large to compute in a practical timeframe?
\end{itemize}
\begin{remark}
If $K=\{K_t\}$ is a filtered complex such that each $K_t$ is also $d$-periodic with respect to the same translational group $T$ then this induces a filtration on each $X_n$ (or $Y_n$).
For any point $(a,b)$ of the persistence diagrams occurring with multiplicity $M(n)$ for $X_n$ (and analogously for $Y_n$) we may bound
\[
M(n)\leq \beta_\bullet^{a,b}(X_{n,t})\leq\beta_\bullet(X_{n,\frac{a+b}{2}})\leq O(n^d)
\]
where $\beta_\bullet^{a,b}(X_{n,t})$ denotes the persistence Betti number between $X_{n,a}$ and $X_{n,b}$.
\end{remark}

If we approximate $K$ with either $X_n$ or $Y_n$, the inaccuracy of the topology occurs due to boundary conditions.
When truncating $K$, irregular cycles appear in $Y_n$ in the outer layers of unit cells. 
When imposing periodic boundary conditions, we instead introduce \textit{toroidal} or \textit{open} cycles which are present in $X_n$ but not present in the lift to $K$.
Toroidal cycles are related to Case~4 of Lemma~\ref{lem:four-cases}, 
and are due to the toroidal structure of $\R^l/ \Z^d$. 
After lifting to $K$, we think of such cycles as representing part of an unbounded $k$-dimensional network through $K$ (i.e.,~infinite paths for $k=1$, infinite sheets for $k=2$ and so on).
For $n$ sufficiently large, toroidal cycles will themselves have translational symmetry within $X_n$, 
so that a cycle may be fixed by some translations $t \in T$ and mapped to distinct copies by other translations. 
This allows us to make a strong statement about the basis of toroidal cycles in $X_n$.

\begin{theorem}{(Toroidal cycle growth)} \label{thm:order}
Let $I^n_\bullet$ be the subgroup of $H_\bullet(X_n)$ induced by the image of the covering map $p_n:K\to X_n$.
Then $\dim(H_\bullet(X_n)/I^n_\bullet)=O(n^{d-1})$ for $n$ sufficiently large.
\end{theorem}

\begin{proof}
Let $\overline{U}$ be a fundamental unit cell for $q(K)$, let $N$ denote the number of simplices in $\overline{U}$ and let 
\[
M=\max_{\mathbf{t}\in T}\max_{i=1,\dots,d}\left\{|c_i|\,:\, \mathbf{t}(\overline{U})\cap \overline{U}\neq\emptyset,\;\mathbf{t}=\sum_{i=1}^dc_i\mathbf{t}_i\right\}
\]
so that $M$ bounds the furthest distance (relative to the fundamental unit) a simplex in $\overline{U}$ may traverse.
For $n>4M$, let $Y_n$ be built from $n^d$ copies of $q(K)$ without periodic boundary conditions and let $Z_n=\bigcup_{c_1,\dots,c_d=2M}^{n-2M-1}(\sum_{i=1}^dc_i\mathbf{t}_i)(\overline{U})$ where $Z_n\cong Y_{n-4M}$ so that $Z_n \subset Y_n$.
Let $\gamma\in C_\bullet(Y_n)$ map to a cycle in $X_n$ (i.e., $p_n(\gamma)\in Z_\bullet(X_n)$).
Then $\D\gamma=0$ or $\D\gamma\in C_\bullet(Y_n)\backslash C_\bullet(Z_n)$.
In either case, if $\Tilde{\gamma}$ also maps to a cycle in $X_n$ and $\D\gamma=\D\Tilde{\gamma}$, then $\gamma$ and $\Tilde{\gamma}$ must represent the same element of $H_\bullet(X_n)/I^n_\bullet$.
This is illustrated in Figure~\ref{fig:toroidal-dimension}, where the difference of the red and blue 1-cycles must necessarily project onto $I^n_1$.

\begin{figure}[ht]
\centering
\begin{tikzpicture}
\fill [green!8] (1,1) rectangle (5,5);
\fill [green!3] (-3.7,1) rectangle (-1,5);
\fill [green!3] (7,1) rectangle (9.7,5);

\draw[step=1cm,gray!30,very thin] (-3.7,-0.7) grid (9.7,6.7);
\draw (0,-0.7) -- (0,6.7);
\draw (6,-0.7) -- (6,6.7);
\draw (-3.7,0) -- (9.7,0);
\draw (-3.7,6) -- (9.7,6);
\draw (2,2) -- (4,2) -- (4,4) -- (2,4) -- (2,2);
\draw [black!30] (-3.7,2) -- (-2,2) -- (-2,4) -- (-3.7,4);
\draw [black!30] (9.7,2) -- (8,2) -- (8,4) -- (9.7,4);

\node[shape=circle] (Yn) at (5.7,0.3) {\textbf{$Y_n$}};
\node[shape=circle] (Zn) at (3.7,2.3) {\textbf{$Z_n$}};

\draw[stealth-stealth] (0,-0.2) -- (6,-0.2);
\node[shape=circle] (n) at (3,-0.4) {$n$};
\draw[stealth-stealth] (1,6.2) -- (2,6.2);
\node[shape=circle] (M) at (1.5,6.4) {\small $M$};

    \draw [>=latex,thick,blue,->] (-0.15,3.5)--(0.5,3.8);
    \draw [>=latex,thick,blue,->] (0.5,3.8)--(0.75,3);
    \draw [>=latex,thick,blue,->] (0.75,3)--(0.3,2);
    \draw [>=latex,thick,blue,->] (0.3,2)--(-0.15,1);
    \draw [>=latex,thick,blue,->] (-0.15,1)--(0.5,0.5);
    \draw [>=latex,thick,blue,->] (0.5,0.5)--(1,1.5);
    \draw [>=latex,thick,blue,->] (1,1.5)--(1.5,2.5);
    \draw [>=latex,thick,blue,->] (1.5,2.5)--(2,1.5);
    \draw [>=latex,thick,blue,->] (2,1.5)--(2.5,2.5);
    \draw [>=latex,thick,blue,->] (2.5,2.5)--(3,3.5);
    \draw [>=latex,thick,blue,->] (3,3.5)--(3,2);
    \draw [>=latex,thick,blue,->] (3,2)--(4.5,0.5);
    \draw [>=latex,thick,blue,->] (4.5,0.5)--(4.5,1.5);
    \draw [>=latex,thick,blue,->] (4.5,1.5)--(5,1);
    \draw [>=latex,thick,blue,->] (5,1)--(6.1,2);
    \draw [>=latex,thick,blue,->] (6.1,2)--(5.7,3);
    \draw [>=latex,thick,blue,->] (5.7,3)--(5.85,3.5);
    
\node[shape=circle,blue] (gamma) at (4.5,1.65) {\small $\gamma$};

    \draw [>=latex,thick,blue!20,->] (5.85,3.5)--(6.5,3.8);
    \draw [>=latex,thick,blue!20,->] (6.5,3.8)--(6.75,3);
    \draw [>=latex,thick,blue!20,->] (6.75,3)--(6.3,2);
    \draw [>=latex,thick,blue!20,->] (6.3,2)--(5.85,1);
    \draw [>=latex,thick,blue!20,->] (5.85,1)--(6.5,0.5);
    \draw [>=latex,thick,blue!20,->] (6.5,0.5)--(7,1.5);
    \draw [>=latex,thick,blue!20,->] (7,1.5)--(7.5,2.5);
    \draw [>=latex,thick,blue!20,->] (7.5,2.5)--(8,1.5);
    \draw [>=latex,thick,blue!20,->] (8,1.5)--(8.5,2.5);
    \draw [>=latex,thick,blue!20,->] (8.5,2.5)--(9,3.5);
    \draw [>=latex,thick,blue!20,->] (9,3.5)--(9,2);
    \draw [thick,blue!20] (9,2)--(9.7,1.3);
    
    \draw [>=latex,thick,blue!20,->] (-3.7,2.1)--(-3.5,2.5);
    \draw [>=latex,thick,blue!20,->] (-3.5,2.5)--(-3,3.5);
    \draw [>=latex,thick,blue!20,->] (-3,3.5)--(-3,2);
    \draw [>=latex,thick,blue!20,->] (-3,2)--(-1.5,0.5);
    \draw [>=latex,thick,blue!20,->] (-1.5,0.5)--(-1.5,1.5);
    \draw [>=latex,thick,blue!20,->] (-1.5,1.5)--(-1,1);
    \draw [>=latex,thick,blue!20,->] (-1,1)--(0.1,2);
    \draw [>=latex,thick,blue!20,->] (0.1,2)--(-0.3,3);
    \draw [>=latex,thick,blue!20,->] (-0.3,3)--(-0.15,3.5);
    
    \draw [>=latex,thick, red!20,->] (-3.7,4)--(-3,5.5);
    \draw [>=latex,thick, red!20,->] (-3,5.5)--(-1.5,6);
    \draw [>=latex,thick, red!20,->] (-1.5,6)--(-0.5,5.5);
    \draw [>=latex,thick, red!20,->] (-0.5,5.5)--(-0.1,4.5);
    \draw [>=latex,thick, red!20,->] (-0.1,4.5)--(-0.15,3.5);
    
    \draw [>=latex,thick, red!20,->] (5.85,3.5)--(6.5,3.2);
    \draw [>=latex,thick, red!20,->] (6.5,3.2)--(6.8,4);
    \draw [>=latex,thick, red!20,->] (6.8,4)--(6.7,4.9);
    \draw [>=latex,thick, red!20,->] (6.7,4.9)--(7.8,5);
    \draw [>=latex,thick, red!20,->] (7.8,5)--(8.3,4);
    \draw [>=latex,thick, red!20,->] (8.3,4)--(9,5.5);
    \draw [thick, red!20] (9,5.5)--(9.7,5.733);

    \draw [>=latex,thick, red,->] (-0.15,3.5)--(0.5,3.2);
    \draw [>=latex,thick, red,->] (0.5,3.2)--(0.8,4);
    \draw [>=latex,thick, red,->] (0.8,4)--(0.7,4.9);
    \draw [>=latex,thick, red,->] (0.7,4.9)--(1.8,5);
    \draw [>=latex,thick, red,->] (1.8,5)--(2.3,4);
    \draw [>=latex,thick, red,->] (2.3,4)--(3,5.5);
    \draw [>=latex,thick, red,->] (3,5.5)--(4.5,6);
    \draw [>=latex,thick, red,->] (4.5,6)--(5.5,5.5);
    \draw [>=latex,thick, red,->] (5.5,5.5)--(5.9,4.5);
    \draw [>=latex,thick, red,->] (5.9,4.5)--(5.85,3.5);
    
\node[shape=circle,red] (gamma) at (1,5.2) {\small $\tilde{\gamma}$};
    
\node[shape=circle,fill=blue,scale=0.3] (d1) at (-0.15,3.5) {};
\node[shape=circle,fill=blue,scale=0.3] (d2) at (5.85,3.5) {};

\path[use as bounding box] (-3.7,-0.7) rectangle (9.7,6.7);
\end{tikzpicture}
\caption{An illustration of $Y_n$ and $Z_n$ for $d=2$ where the arrow labelled $n$ (resp. $M$) indicates a width of $n$ ($M$) fundamental domains. The chains representing $\gamma$ and $\Tilde{\gamma}$ have common boundary. They represent toroidal cycles of $H_1(X_n)$ and will represent the same class of $H_1(X_n)/I^n_1$. Moreover, $\D\gamma$ is not in $C_\bullet(Z_n)$ as $Z_n$ is strictly contained in the green region.}
\label{fig:toroidal-dimension}
\end{figure}
Now, fix a basis $\mathcal{B}$ of $H_\bullet(X_n)/I^n_\bullet$.
For every basis element $\beta_i\in\mathcal{B}$, choose a representative chain $\gamma_i\in C_\bullet(Y_n)$ and map $\beta_i\mapsto \D\gamma_i \mod C_\bullet(Z_n)$.
The argument above ensures that the $\gamma_i$ will be linearly independent, so this map extends linearly to an injective homomorphism into $C_\bullet(Y_n)/C_\bullet(Z_n)$.
The dimension of the codomain is bounded by the number of cells in $Y_n$ not contained in $Z_n$, so
\[
\dim(H_\bullet(X_n)/I^n_\bullet)\leq \dim(C_\bullet(Y_n)/C_\bullet(Z_n)) \leq N\left[n^d-(n-4M)^d\right] = O(n^{d-1}).
\]
\end{proof}


It is also possible for translational equivalence classes of non-toroidal cycles to appear with multiplicity $O(n^{d-1})$.
For example, consider a 1-periodic infinite cylinder, $K_C$, with $\beta_1(K_C)=1$.
Then for each $n$, $X_n\cong\T^2$  contains one toroidal and one non-toroidal cycle, so $\beta_1(X_n)=O(1)$ despite the existence of the non-toroidal 1-cycle.
In this case, the non-toroidal cycle (the circle of the cylinder) is, in a sense, coupled to the toroidal cycle (the torus formed by gluing opposite ends of the cylinder).
In fact, we will always be able to construct a toroidal cycle whenever we observe $O(n^{d-1})$ growth when studying families of non-toroidal cycles\footnote{Although this will not necessarily tell us about the size of the basis}.

\begin{theorem} \label{thm:order-nec}
Suppose $\gamma\in Z_k(X_n)$ is a non-toroidal $k$-cycle and $\{[\tr(\gamma)]\,:\,\tr\in T/nT\}$ is a family of $O(n^{d-m})$ non-toroidal homology classes for some positive integer $m\leq d$.
Furthermore, assume that $\tr(\gamma)\neq\gamma$ for any $\tr\in T/nT$.
Then there exist toroidal $(k+1)$-cycles $\alpha_1,\dots,\alpha_m\in Z_{k+1}(X_n)$ such that $\{[\tr(\alpha_i)]\,:\,\tr\in T/nT\}$ for $i=1,\dots,m$ are families of $O(n^{d-m})$ classes of $H_{k+1}(X_n)/I^n_k$.
\end{theorem}


\begin{proof}
There exists independent translations $\tr_1,\dots,\tr_m\in T$ such that $\tilde{\gamma}\sim\tr_i(\tilde{\gamma})$ for $i=1,\dots,m$ and $\tilde{\gamma}\in Z_k(K)$ an appropriate lift of $\gamma$.
Take $\tilde{\alpha}_i$ so that $\D(\tilde{\alpha}_i)=\tilde{\gamma}-\tr_i(\tilde{\gamma})$.
Then $\tilde{\alpha}_i\neq0$ and $\alpha_i=q_n\left(\sum_{l=0}^{n_i}\tr_i^{l}(\tilde{\alpha}_i)\right)$ is a toroidal cycle in $X_n$ where $n_i=\min(l\,:\,\tr_i^l\in nT)$.
Then $\tr_i(\alpha_i)\sim\alpha_i$ for $1\leq i\leq m$.
Otherwise, if $i\neq j$, observe that the alternating sum $\tr_j(\tilde{\alpha}_i)-\tr_i(\tilde{\alpha}_j)+\tilde{\alpha}_i-\tilde{\alpha}_j:=\lambda_{i,j}\in Z_{k+1}(K)$ is a commutator cycle in $K$.
Hence 
\[
\tr_j(\alpha_i)-\alpha_i
= \sum_{l=0}^{n_i}q_n\left(\tr_j\circ \tr_i^l(\tilde{\alpha}_i) - \tr_i^l(\tilde{\alpha}_i)\right)
= \sum_{l=0}^{n_i}q_n\left(\tr_i^l(\lambda_{i,j})\right)
\]
is non-toroidal.
Thus, for each $1\leq i\leq m$, $\tr(\alpha_i)$ will represent the same class of $H_{k+1}(X_n)/I^n$ for each $\tr\in\langle \tr_1,\dots,\tr_m\rangle$.
\end{proof}

It is important to note, as was mentioned by one of our reviewers, that the condition that $\tr(\gamma)\neq\gamma$ is necessary for this Theorem to hold
as it ensures that each $\tilde{\alpha}_i$ is non-zero so that $\alpha_i$ is toroidal.
For example, if $K \subset \R^2$ is the set of integer coordinate points $\Z^2$ and $T$ is the maximal translation group, then the non-toroidal cycle of $X_n$ represented by the sum of all (isolated) vertices is translation invariant but the construction in the proof of Theorem~\ref{thm:order-nec} fails to identify any toroidal 1-cycles.
Indeed, for this example the only 1-cycle is the zero cycle.
However, note that even for this ``counterexample'' we can write this cycle as the sum of individual vertices, whose translations are all non-homologous non-toroidal cycles.
There are exactly $n^2$ of these vertices, which is in line with the heuristic we are building for non-toroidal cycle growth with $n$.

The method of generating toroidal cycles from lower dimensional non-toroidal cycles described in Theorem~\ref{thm:order-nec} helps us determine a large collection of other cycles, both toroidal or non-toroidal.

\begin{corollary} \label{cor:order}
Under the same hypotheses as in Theorem~\ref{thm:order-nec}, for every non-empty $L\subset\{1,\dots,m\}$ with $\ell=|L|$ there exists a trichotomy
\begin{enumerate}
\item There exists a corresponding toroidal $(k+\ell)$-cycle, $\alpha_L\in Z_{k+\ell}(X_n)$, such that $\{[\tr(\alpha_L)]\,:\,\tr\in T/nT\}$ is a family of $O(n^{d-m})$ classes of $H_{k+\ell}(X_n)/I^n$.
\item There exists a corresponding non-toroidal $(k+\ell-1)$-cycle, $\alpha_L\in Z_{k+\ell-1}(X_n)$, such that $\{[\tr(\alpha_L)]\,:\,\tr\in T/nT\}$ is a family of $O(n^d)$ classes of $I^n$.
\item Some proper subset of $L$ satisfies (2).
\end{enumerate}
\end{corollary}

\begin{proof}
Apply Theorem~\ref{thm:order-nec} inductively, with base case $\ell=1$.
For $\ell>1$ and toroidal cycles associated with $\tr_{L_1},\dots,\tr_{L_{\ell-1}}$, take the alternating sum of these with $\tr_{L_\ell}$, as was done in the proof of Theorem~\ref{thm:order-nec} to construct $\lambda_{i,j}$.
This either generates another family of $O(n^{d-m})$ classes of $H_{k+\ell}(X_n)/I^n$ if the alternating sum of $(k+\ell-1)$-cycles is a boundary, or a family of $O(n^d)$ non-toroidal $(k+\ell-1)$-cycles otherwise.
If we identify non-toroidal $(k+\ell-1)$-cycles then our induction is complete, as we are unable to generate higher-dimensional cells with this method.
This shows (1) and (2), and for (3) we note that Theorem~\ref{thm:order-nec} implies the existence of subsets of $L$ satisfying (1), and we need only apply the above inductive process to completion along a given sequence to gain a subset of $L$ satisfying (2).
\end{proof}

We interpret these relationships between toroidal and non-toroidal cycles which arise from Theorem~\ref{thm:order-nec} and Corollary~\ref{cor:order} as a decomposition of a toroidal $k$-cycle into a non-toroidal $(k-m)$-cycle appearing with multiplicity $O(n^{d-m})$ and an $m$-cycle of the ``ambient torus'' created by the $d$-periodic structure.
We conjecture that this is true of all toroidal cycles.


\section{The Mayer-Vietoris Spectral Sequence}
\label{sec:mvss}

In this section, $K$ is as in Section~\ref{sec:Complexes}, and for any differential graded module $(M,d)$ such that $\im(d)< \ker(d)$, we will denote $H(M,d):=\ker(d)/\im(d)$ to be its homology or $H(M)$ when the differential $d$ is clear from the context (e.g., the boundary map).
If the reader is already familiar with the Mayer-Vietoris spectral sequence, they may wish to skip to Section~\ref{sec:mvss-2}.

\subsection{Construction} \label{sec:mvss-def}

Here we give an overview of the construction of the Mayer-Vietoris spectral sequence.
We direct the reader to \cite{brown,stafa2015mayer} for further details (or to \cite{lewis2014multicore} for an interpretation with tensors).

Let $\U=\{U_{i\in I}\}$ be a cover of $K$ by subcomplexes such that only finite intersections of sets in $\U$ may be non-empty, and let $\nerve$ denote the nerve of $\U$.
The \textit{blow-up} complex of $K$ with respect to $\U$ is the $\Z$-bigraded module $E^0=\{E^0_{p,q}\}_{p,q\in\Z}$, where 
\[
E^0_{p,q}:= \langle (J,\gamma)\,:\,\gamma\in C_q\left(\cap_{j\in J}U_j\right),\,|J|=p+1\rangle
\]
is a subspace of the tensor product $C_p(\nerve)\otimes C_q(K)$ representing $q$-chains appearing in $(p+1)$-fold intersections of $\U$
(here we use $(V,\gamma)$ to denote $V\otimes\gamma$).
$E^0$ is equipped with the maps $\D^0_{p,q}:E^0_{p,q}\to E^0_{p,q-1}$ and $\D^1_{p,q}:E^0_{p,q}\to E^0_{p-1,q}$ induced by the boundary maps on $K$ and $\nerve$ respectively.
Explicitly: $\D^0_{p,q}(x,y)=(x,\D y)$, $\D^1_{p,q}(x,y)=(-1)^q(\D x,y)$ where $\D$ is the usual boundary map/s and both maps extend linearly to $E^0_{p,q}$. 


The maps $\D^0$ and $\D^1$ satisfy  $(\D^0)^2= 0, (\D^1)^2=0, \D^0\D^1+\D^1\D^0=0$, which makes the blow-up complex $(E^0,\D^0,\D^1)$ a \textit{bi-complex} --- a two-parameter analogue of a differential graded complex (e.g., a chain complex).
Algebraically, we picture the blow-up complex with the diagram below, where every square \emph{anticommutes}.
\begin{center}
\begin{tikzpicture}
\node (00) at (0,0) {$E^0_{00}$};
\node (10) at (2,0) {$E^0_{10}$};
\node (20) at (4,0) {$E^0_{20}$};
\node (01) at (0,1.5) {$E^0_{01}$};
\node (11) at (2,1.5) {$E^0_{11}$};
\node (21) at (4,1.5) {$E^0_{21}$};
\node (02) at (0,3) {$E^0_{02}$};
\node (12) at (2,3) {$E^0_{12}$};
\node (22) at (4,3) {$E^0_{22}$};
\node (m10) at (-1,0) {$0$};
\node (m11) at (-1,1.5) {$0$};
\node (m12) at (-1,3) {$0$};
\node (30) at (6,0) {$\cdots$};
\node (31) at (6,1.5) {$\cdots$};
\node (32) at (6,3) {$\cdots$};
\node (0m1) at (0,-0.75) {$0$};
\node (1m1) at (2,-0.75) {$0$};
\node (2m1) at (4,-0.75) {$0$};
\node (03) at (0,4.5) {$\vdots$};
\node (13) at (2,4.5) {$\vdots$};
\node (23) at (4,4.5) {$\vdots$};

\draw[>=latex,->] (00) edge node[right] {} (0m1);
\draw[>=latex,->] (00) edge node[above] {} (m10);
\draw[>=latex,->] (10) edge node[right] {} (1m1);
\draw[>=latex,->] (10) edge node[above] {\small $\D^1$} (00);
\draw[>=latex,->] (20) edge node[right] {} (2m1);
\draw[>=latex,->] (20) edge node[above] {\small $\D^1$} (10);
\draw[>=latex,->] (30) edge node[above] {\small $\D^1$} (20);

\draw[>=latex,->] (01) edge node[right] {\small $\D^0$} (00);
\draw[>=latex,->] (01) edge node[above] {} (m11);
\draw[>=latex,->] (11) edge node[right] {\small $\D^0$} (10);
\draw[>=latex,->] (11) edge node[above] {\small $\D^1$} (01);
\draw[>=latex,->] (21) edge node[right] {\small $\D^0$} (20);
\draw[>=latex,->] (21) edge node[above] {\small $\D^1$} (11);
\draw[>=latex,->] (31) edge node[above] {\small $\D^1$} (21);

\draw[>=latex,->] (02) edge node[right] {\small $\D^0$} (01);
\draw[>=latex,->] (02) edge node[above] {} (m12);
\draw[>=latex,->] (12) edge node[right] {\small $\D^0$} (11);
\draw[>=latex,->] (12) edge node[above] {\small $\D^1$} (02);
\draw[>=latex,->] (22) edge node[right] {\small $\D^0$} (21);
\draw[>=latex,->] (22) edge node[above] {\small $\D^1$} (12);
\draw[>=latex,->] (32) edge node[above] {\small $\D^1$} (22);

\draw[>=latex,->] (03) edge node[right] {\small $\D^0$} (02);
\draw[>=latex,->] (13) edge node[right] {\small $\D^0$} (12);
\draw[>=latex,->] (23) edge node[right] {\small $\D^0$} (22);
\end{tikzpicture}
\end{center}

A \textit{spectral sequence} $(E^r,d^r)$ is a collection of $\Z$-bigraded modules $E^r$ and differential maps $d^r$ with the property that $H(E^r,d^r)=E^{r+1}$ (that is, the homology with respect to $d^r$ of $E^r$ determines $E^{r+1}$).
We say that $E^r$ is the $r^\mathrm{th}$-\textit{page} of the spectral sequence.
If for each $p,q\in\Z$ there exists an $r_{p,q}$ such that $E^r_{p,q}\cong E^{r_{p,q}}_{p,q}$ for $r \geq r_{p,q}$, then we define $E^\infty=\{E^{r_{p,q}}_{p,q}\}_{p,q\in\Z}$ to be the $\infty$-\textit{page} of the spectral sequence and say that $(E^r,d^r)$ converges to $E^\infty$.
In most cases, the $r_{p,q}$ achieve a maximum $R$ and we set $E^\infty:=E^R$.

Every bi-complex $(M,d^0,d^1)$ induces a spectral sequence for which $M=E^0$, $E^1=H(M,d^0)$ and $E^2=H(E^1,d^1)$.  
In the latter, we have overloaded the notation by writing $d^1$ for the map on the quotient space $E^1 = H(E^0,d^0)$ induced by the second differential map $d^1:E^0\to E^0$.
Finally, the \textit{Mayer-Vietoris spectral sequence} (MVSS) of a covering $\U$ of $K$ is the spectral sequence induced by the blow-up complex, $(E^0, \D^0, \D^1)$, whose diagonals can be identified with the homology of $K$, as we will soon see. This means 
\begin{itemize}
    \item $E^1=H(E^0,\D^0)$ is the combined (cellular) homology of the components of the cover.
    \item $E^2=H(E^1,\D^1)$ is induced by the (simplicial) homology of the nerve.
    \item For $r>1$, $\D^r=\D^1(\D^0)^{-1}\D^{(r-1)}:E^r_{p,q}\to E^r_{p-r,q+r-1}$ inductively determines the differential map on the $r^\mathrm{th}$ page whose image is independent of the choice of representative (c.f. \cite{lipsky2011parallelized}) when applying $(\D^0)^{-1}$ (see the following diagram). 
    \item $(E^r,\D^r)$ converges (with  $r_{p,q}=\max(0,p+1,q+2)$).
\end{itemize}

\begin{center}
\begin{tikzpicture}[scale = 0.95]
\node (1) at (6,-1.5) {$\cdots$};
\node (2) at (4,-1.5) {$E^0_{p,q-1}$};
\node (3) at (4,0) {$E^0_{p,q}$};
\node (4) at (2,0) {$E^0_{p-1,q}$};
\node (5) at (2,1.5) {$E^0_{p-1,q+1}$};
\node (6) at (0,1.5) {$E^0_{p-2,q+1}$};
\node (7) at (0,3) {$E^0_{p-2,q+2}$};
\node (8) at (-2,3) {$E^0_{p-3,q+2}$};
\node (9) at (-2,4.5) {$\vdots$};

\draw[>=latex,->] (1) edge node[above] {\tiny $\D^1$} (2);
\draw[>=latex,->,blue,line width=0.3mm] (3) edge node[left] {\small $\D^0$} (2);
\draw[>=latex,->,red,line width=0.3mm] (3) edge node[above] {\small $\D^1$} (4);
\draw[>=latex,->] (5) edge node[left] {\tiny $\D^0$} (4);
\draw[>=latex,->] (5) edge node[above] {\tiny $\D^1$} (6);
\draw[>=latex,->] (7) edge node[left] {\tiny $\D^0$} (6);
\draw[>=latex,->] (7) edge node[above] {\tiny $\D^1$} (8);
\draw[>=latex,->] (9) edge node[left] {\tiny $\D^0$} (8);

\draw[>=latex,->,orange,line width=0.3mm] (3) edge[out=-155,in=-70] node[below] {\small $\D^2$} (6);
\draw[>=latex,->,teal,line width=0.3mm] (3) edge[out=-140,in=-70] node[below] {\small $\D^3$} (8);
\end{tikzpicture}
\end{center}

The Mayer-Vietoris spectral sequence is so named because it generalises the standard exact sequence to covers with more than two elements. 
We can determine the homology of $K$ from the diagonals of $E^\infty$.
The simplest non-trivial scenario is when $\U=\{A,B\}$, $K = A \cup B$, and $\nerve = \{A, B, A \cap B \}$.  
The standard Mayer-Vietoris exact sequence is 
\[
\cdots \longrightarrow H_q(A\cap B) \xrightarrow{(i^*,j^*)}H_q(A)\oplus H_q(B) \xrightarrow{k^*-l^*}H_q(K)\xrightarrow{\;\delta\;} H_{q-1}(A\cap B) \longrightarrow \cdots
\]
The blow-up complex has just two columns of chain complexes $E^0_{0,\bullet}$ and $E^0_{1,\bullet}$ and the spectral sequence will converge on the second page. 
To see this, 
we note $E^1_{0,q} = H_q(A)\oplus H_q(B)$ and $E^1_{1,q} = H_q(A\cap B)$.
We may then identify the kernel and image of  $\D^1:E^1_{1,q}\to E^1_{0,q}$ as the kernel and image (resp.) of the map $H_q(A\cap B)\xrightarrow{(i^*,j^*)}H_q(A)\oplus H_q(B)$ in the exact sequence.
And from the exact sequence we determine the homology of $K$ by
\begin{multline*}
H_q(K) \cong \mathrm{coker}\left(H_q(A\cap B)\xrightarrow{(i^*,j^*)}H_q(A)\oplus H_q(B)\right) \\
    \oplus \mathrm{ker}\left(H_{q-1}(A\cap B)\xrightarrow{(i^*,j^*)}H_{q-1}(A)\oplus H_{q-1}(B)\right)
\end{multline*}
which is equivalent to the expression $H_k(K)\cong E^\infty_{0,k}\oplus E^\infty_{1,k-1}$ in the spectral sequence notation.
This generalises to a more fundamental result of \cite{godement1958topologie}, where $E^\infty_{0,k}\oplus E^\infty_{1,k-1}$ extends to be the direct sum of off-diagonal terms.
To make this explicit, we identify the \textit{total complex} of $E^0$, denoted $T_\U$, to be the graded complex of diagonals from the blow-up complex $E^0$.
That is,
$(T_\U)_k:=\bigoplus_{p+q=k}E^0_{p,q}$.
The fact that squares in the above diagram anticommute means $(\D^0+\D^1)^2=0$, making $(T_\U,\D^0+\D^1)$ a differential graded complex.

\begin{theorem}
\[
H_k(K)\cong H_k(T_\U,\D^0+\D^1)\cong\bigoplus_{p+q=k}E^\infty_{p,q}
\]
\end{theorem}

Explicitly, $H_k(T_\U,\D^0+\D^1)$ permits a filtration $\{H_k(T_\U,\D^0+\D^1)^i\,:\,i=0,\dots,k\}$ for which $E^\infty_{p,k-p}\cong H_k(T_\U,\D^0+\D^1)^p/H_k(T_\U,\D^0+\D^1)^{p-1}$.
The right-most isomorphism of the theorem applies in general only for vector spaces and free modules.
When we generalise to other modules (for instance persistence modules over the ring $\F[t]$), we run into an ``extension problem'' in how to determine $H_k(T_\U,\D^0+\D^1)^p$ from $E^\infty_{p,k-p}$ and $H_k(T_\U,\D^0+\D^1)^{p-1}$.

The isomorphism $H_k(T_\U,\D^0+\D^1)\to H_k(K)$ is induced by the (sum of) inclusion map(s) from $E^0_{0,k}$ into $C_k(K)$ (we denote these maps by $i^*_U$ below).
This relates a cycle in $K$ to a ``blow-up'' of its cells in each component of the cover (the $E^0_{0,k}$ component of $(T_\U)_k$) and the lower-dimensional intersections of these cells occurring in the higher-dimensional components of the cover (the $E^0_{p,k-p}$ components of $(T_\U)_k$ for $p>0$).
On the other hand, the isomorphism $H_k(T_\U,\D^0+\D^1)\to\bigoplus_{p+q=k}E^\infty_{p,q}$ describes a filtration of $H_k(T_\U,\D^0+\D^1)$, whereby $H_k(E^0_{0,k},\D^0+\D^1)\cong E^\infty_{0,k}$ and inductively $E^\infty_{p,k-p}\cong H_k(\bigoplus_{i=0}^pE^0_{i,k-i},\D^0+\D^1)/H_k(\bigoplus_{i=0}^{p-1}E^0_{i,k-i},\D^0+\D^1)$ for $0<p\leq k$.
We think of an element $\gamma\in E^\infty_{p,k-p}$ existing only due to $p$-fold intersections of the cover as it corresponds to a class of $H_k(T_\U,\D^0+\D^1)$ with representative cycles in $\bigoplus_{i=0}^pE^0_{i,k-i}$ but none in $\bigoplus_{i=0}^{p-1}E^0_{i,k-i}$.
\[
[\alpha_p]\in E^\infty_{p,k-p} \;\Longrightarrow\; \exists\, [(\alpha_0,\dots,\alpha_p,0,\dots,0)]\in H_k(T_\U,\D^0+\D^1) \;\Longleftrightarrow\;\left[\sum_{U\in\U}i_U^*(\alpha_0)\right] \in H_k(K)
\]

The benefits of the MVSS in application are two-fold.
First, that the isomorphism $H_k(K)\to\bigoplus_{p+q=k}E^\infty_{p,q}$ decomposes in such a way that a given basis of cycles is minimal with respect to intersections of the cover.
This is explicitly due to the above correspondence, the left implication is an equivalence whenever $p$ is the minimal value for which there exists $[(\alpha_0,\dots,\alpha_p,0,\dots,0)]\in H_k(T_\U,\D^0+\D^1)$.
Second, that we may replace one large (possibly infinite) calculation of homology for $H_\bullet(K)$ with many smaller calculations done in parallel.
The latter allows us to think of the MVSS as an algorithm for difficult homology calculations, where each page can be considered a refinement of $H_\bullet(K)$.

\subsection{Application to Periodic Cellular Complexes} \label{sec:mvss-2}

A partial solution to the problem of identifying toroidal cycles involves introducing a Mayer-Vietoris spectral sequence to calculate the homology.
Given a fundamental unit cell, $U$, observe that
\[
\U= \{\mathbf{t}(\overline{U})\,:\,\mathbf{t}\in T\}
\]
is a covering of $K$ by subcomplexes, so induces a blow-up complex $(E^0,\partial_0,\partial_1)$ of $K$ from which $H_\bullet(K)$ may be calculated.
There are still infinitely many cells and so we cannot implement homology calculations with matrices.
However, all homology calculations of the blow-up complex are locally defined and identical up to translation, offering the benefit of calculating the MVSS via many smaller parallel calculations on finite unit cells on each page\footnote{In principle, $\overline{U}$ may be replaced by any subcomplex whose periodic copies cover $K$.}.

Recall that $X_n$ denotes the bounded space of $n\times\cdots\times n$ copies of $q(K)$ with periodic boundary conditions.
An analogous construction creates a MVSS for $X_n$, which may instead be calculated with a finite algorithm.
The toroidal cycles of $X_n$ will be elucidated on the second (and higher) page(s) of this MVSS as a result of applying the nerve boundary map.

These cycles will intersect many components of the cover $\U$ and therefore be non-localised in $X_n$ -- reconstructed using higher-dimensional components of $\nerve(\U)$ and therefore in the right-most columns of $E^\infty$.
Conversely, we expect ``true'' cycles of $K$ to be localised in $X_n$,  
and so it is possible for some of these cycles to be reconstructed with lower-dimensional components of $\nerve(\U)$.

\begin{proposition}
If $\gamma\in H_\bullet(X_n)$ can be identified with an element of $E^\infty_{0\bullet}$ then $\gamma$ is a non-toroidal cycle.
\end{proposition}
\begin{proof}
If $\gamma$ corresponds to an element of $E^\infty_{0\bullet}$ then $\gamma$ in turn corresponds to an element of $H_\bullet(T_\U,\D^0+\D^1)$ with represenative in $E^0_{0\bullet}$.
But then $\gamma$ must be a sum of cycles, each of which is contained in a single copy of $\overline{U}$ and necessarily lifts to $K$.
\end{proof}

We now have two necessary 
conditions for identifying toroidal cycles of $X_n$ via a MVSS -- those that appear in a collection of \textit{strictly fewer} than $O(n^d)$ translationally equivalent cycles which also do \textit{not} correspond to an element of $E^\infty_{0\bullet}$.
These conditions provide a heuristic with which we can identify potential toroidal cycles.
We use ``heuristic'' in a similar sense to the heuristic in persistent homology that cycles with a large persistence are ``significant''.
In our case ``toroidal cycle'' has a precise definition, so once we have identified a potential candidate it is possible to test if it truly is toroidal. 
We leave the implementation of an explicit algorithm to automate this verification step to future work.

\subsection{Intersecting Planes Example}

In this section we demonstrate a relatively simple synthetic example for how to use the heuristic for identifying toroidal cycles.
This example is constructed to avoid false positive detections, and likely does not resemble real data.
However the example demonstrates how to interpret the MVSS and extract the corresponding homology basis to verify if a cycle is toroidal or not.
In particular, if one encounters a situation where $O(n^{d-1})$ toroidal generators are dominated by $\dim(E^{\infty}_{pq})=O(n^d)$ then the basis extraction may help identify this.

\begin{figure}[ht]
\centering
\begin{tabular}{ccc}
\includegraphics{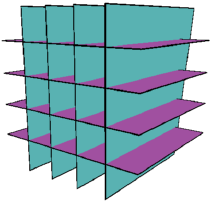} $\quad \quad$
&
\includegraphics[scale=0.8]{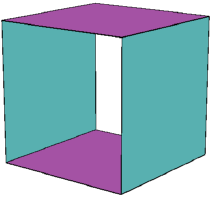}  $\quad \quad$
&
\includegraphics[scale=0.8]{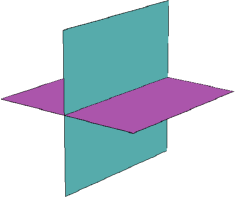}  $\quad \quad$
\end{tabular}
\caption{\emph{Left:} A 3-periodic cell-complex, $K$, constructed from two sets of orthogonal planes; \emph{center:} a choice of fundamental unit cell, $V$, used to construct a MVSS; \emph{right:} another choice of fundamental unit cell discussed at the end of the section. }
\label{fig:intersecting_planes}
\end{figure}

Consider two orthogonal planes $P_1$ and $P_2$ in $\R^3$ with orthonormal vectors $\mathbf{n}_1$ and $\mathbf{n}_2$ respectively.
Let $K=\bigcup_{n\in\Z}(P_1+n\mathbf{n}_1)\cup\bigcup_{m\in\Z}(P_2+m\mathbf{n}_2)$ so that $K$, Figure~\ref{fig:intersecting_planes} (left), is a $3$-periodic collection of square cylinders whose translation group, $T$, is generated by shifts $\mathbf{e}_1,\mathbf{e}_2,\mathbf{e}_3$ along $\mathbf{n}_1,\mathbf{n}_2$ and $\mathbf{n}_1\times\mathbf{n}_2$ respectively.
Without loss of generality, we take $\mathbf{n}_1,\mathbf{n}_2$ and $\mathbf{n}_1\times\mathbf{n}_2$ to be the standard basis vectors.

$K$ has a cellular decomposition into planar squares.
We choose the fundamental domain $[0,1)^3$ of $\R^3$ with respect to $T$ and let $V=K\cap[0,1]^3$.
$V$ resembles a $3$-cube with no interior and two sides removed as shown in Figure~\ref{fig:intersecting_planes} (center). 

Let $X_n$ be the union of $n\times n\times n$ ($n>2$) adjacent copies of $V$ with periodic boundary conditions imposed, and cover $X_n$ with (translated) copies of $V$.
Taking coefficients in a field $\F$, the non-trivial part of the blow-up complex $E^0$ will be as follows.

\begin{center}
\begin{tikzpicture}
\node (0,0) at (0,0) {$\F^{8n^3}$};
\node (1,0) at (1,0) {$\F^{28n^3}$};
\node (2,0) at (2,0) {$\F^{56n^3}$};
\node (3,0) at (3,0) {$\F^{70n^3}$};
\node (4,0) at (4,0) {$\F^{56n^3}$};
\node (5,0) at (5,0) {$\F^{28n^3}$};
\node (6,0) at (6,0) {$\F^{8n^3}$};
\node (7,0) at (7,0) {$\F^{n^3}$};
\node (0,1.5) at (0,1.5) {$\F^{12n^3}$};
\node (1,1.5) at (1,1.5) {$\F^{18n^3}$};
\node (2,1.5) at (2,1.5) {$\F^{12n^3}$};
\node (3,1.5) at (3,1.5) {$\F^{3n^3}$};
\node (0,3) at (0,3) {$\F^{4n^3}$};
\node (1,3) at (1,3) {$\F^{2n^3}$};

\path[every node/.style={font=\sffamily\small}]
        \foreach \i in {0,...,3}
        {
        (\i,1.2) edge[>=latex,->] node[right] {\tiny $\partial^0$} (\i,0.3)
        }
        \foreach \j in {0,1}
        {
        (\j,2.7) edge[>=latex,->] node[right] {\tiny $\partial^0$} (\j,1.8)
        };
        
\end{tikzpicture}
\end{center}

A simple cellular homology calculation gives the $E^1$ page below.

\begin{center}
\begin{tikzpicture}
\node (0,0) at (0,0) {$\F^{n^3}$};
\node (1.5,0) at (1.5,0) {$\F^{13n^3}$};
\node (3,0) at (3,0) {$\F^{44n^3}$};
\node (4.5,0) at (4.5,0) {$\F^{67n^3}$};
\node (6,0) at (6,0) {$\F^{56n^3}$};
\node (7.5,0) at (7.5,0) {$\F^{28n^3}$};
\node (9,0) at (9,0) {$\F^{8n^3}$};
\node (10.5,0) at (10.5,0) {$\F^{n^3}$};
\node (0,1) at (0,1) {$\F^{n^3}$};
\node (1.5,1) at (1.5,1) {$\F^{n^3}$};

\path[every node/.style={font=\sffamily\small}]
        \foreach \i in {1,...,7}
        {
        (1.5*\i-0.4,0) edge[>=latex,->] node[above] {\tiny $\partial^1$} (1.5*\i-1.1,0)
        }
        (1.1,1) edge[>=latex,->] node[above] {\tiny $\partial^1$} (0.4,1);
\end{tikzpicture}
\end{center}

$(E^1_{\bullet 0},\D_1)$ is exactly the chain complex of the nerve $\nerve$ of the cover by $\{\mathbf{t}(V)\}$.
$\nerve$ is a flag complex, where for $\mathbf{t}=\sum_{i=1}^3t_i\mathbf{e}_i$ and $\mathbf{s}=\sum_{i=1}^3s_i\mathbf{e}_i$, an edge joins the vertices representing $\mathbf{t}(V)$ and $\mathbf{s}(V)$ if and only if $\max|t_i-s_i|=1$.
In this way, we may identify $\nerve$ with the nerve of the cover of $\R^3/(n\Z)^3$ by unit cubes.
Each cube and its intersections with other cubes is acyclic for $n>3$, so $E^2_{\bullet 0}$ is given by $H(E^1_{\bullet 0},\D^1)\cong H_\bullet(\R^3/(n\Z)^3) \cong H_\bullet(\T^3)$.
Next, $E^1_{p, 1}\cong \F[(\Z/n\Z)^3]$ for $p=0,1$ where $\F[(\Z/n\Z)^3]$ denotes the $\F$-vector space with basis $(\Z/n\Z)^3$. The isomorphism identifies the sole generator of the homology of the $(i,j,k)^\mathrm{th}$ copy of $V$ with the basis element $(i,j,k)$.
This identification associates $\D_1:E^1_{11}\to E^1_{01}$ with the map $(i,j,k)\mapsto(i,j,k)+(i+1,j,k)$ which has rank $n^3-n^2$.
\begin{center}
\begin{tikzpicture}
\node (label) at (-3,0.5) {$E^2\,:$};
\node (0,0) at (0,0) {$\F$};
\node (1,0) at (1,0) {$\F^3$};
\node (2,0) at (2,0) {$\F^3$};
\node (3,0) at (3,0) {$\F$};
\node (0,1) at (0,1) {$\F^{n^2}$};
\node (1,1) at (1,1) {$\F^{n^2}$};

\path[every node/.style={font=\sffamily\small}]
        \foreach \i in {2,3}
        {
        (\i-0.25,0.05) edge[>=latex,->] node[above] {\tiny $\partial^2$} (\i-1.85,0.8)
        };
\end{tikzpicture}
\end{center}

Finally, we observe that $E^3=E^\infty$.
To calculate $E^3$ we note that $\D^2:E^2_{30}\to E^2_{11}$ must be injective, as otherwise this implies $H_3(X_n)\neq 0$ despite $X_n$ being a $2$-dimensional complex.
Note also that for a fixed choice of vertex $v$ in $V$ and the permutation $\varsigma$ which maps $\varsigma(i,j,k)=(k,j,i)$, the three classes of $E^2_{20}$ of the form
\[
\sum_{i,j} \left[ \left(\tr_{\varsigma^{\ell}(i,j,0)}(v),\left\{\tr_{\varsigma^{\ell}(i,j,0)},\tr_{\varsigma^{\ell}(i,j+1,0)},\tr_{\varsigma^{\ell}(i+1,j,0)}\right\}\right) - \left(\tr_{\varsigma^{\ell}(i,j,0)}(v),\left\{\tr_{\varsigma^{\ell}(i,j+1,0)},\tr_{\varsigma^{\ell}(i+1,j,0)},\tr_{\varsigma^{\ell}(i+1,j+1,0)}\right\}\right) \right]
\]
for $\ell=0,1,-1$ form a basis.
After diagram chase, each basis element will map under $\D^2$ to the sum over each copy of $V$ of a single square face (where any pair of basis elements will map to squares in orthogonal planes).
Only one of these three generators will have non-trivial image (the one whose square faces are not the boundary of a 2-cell) and so $\D^2$ has rank 1.
Thus the MVSS converges to the following complex.
\begin{center}
\begin{tikzpicture}
\node (label) at (-3,0.5) {$E^\infty\,:$};
\node (0,0) at (0,0) {$\F$};
\node (1.5,0) at (1.5,0) {$\F^3$};
\node (3,0) at (3,0) {$\F^2$};
\node (0,1) at (0,1) {$\F^{n^2-1}$};
\node (1.5,1) at (1.5,1) {$\F^{n^2-1}$};
\end{tikzpicture}
\end{center}

\begin{itemize}
    \item $E^\infty_{00}$ is identified with the single connected component of $X_n$, which will always be non-toroidal and tells us that $K$ is connected since this is true for each $n\in\N$.
    \item $E^\infty_{01}$ is identified with the square $1$-cycles in each copy of $V$ as discussed above, which again must be true cycles of $K$ and lift in the obvious way.
    \item $E^\infty_{10}$ is identified with three 1-cycles generated by orthogonal paths through $X_n$. These are toroidal and lift to construct the $1$-dimensional lattice framework of $K$ in the standard basis directions. In particular, these generate the 1-homology of the ambient torus.
    \item $E^\infty_{11}$ is identified with square tori 2-cycles in $X_n$.  $E^\infty_{11}$ is associated with the same square cylinders that generate $E^\infty_{01}$, except now these are viewed as sitting in the intersection of two adjacent copies of $V$. These in turn correspond to infinite square cylinders in $K$ which project onto square tori in $X_n$. There are $n^2-1$ cycles by analogy to $E^\infty_{01}$, all of which are toroidal.
    \item $E^\infty_{20}$ is represented in $E^0_{20}$ by the sum of vertices in three-fold intersections of co-planar copies of $V$. In $K$ these correspond to the collection of orthogonal planes (purple and cyan) which can be identified with open 2-tori in $X_n$. These are toroidal cycles. There are a constant number of generators of $E^\infty_{20}$ for similar reasons to $E^\infty_{10}$, and likewise correspond to generators of the 2-homology of the ambient torus.
\end{itemize}

The toroidal 2-cycles associated with $E^\infty_{11}$ are coupled to the $O(n^2)$ non-toroidal cycles associated with $E^\infty_{10}$.
The 2-cycles of the square tori associated to the classes of $E^\infty_{11}$ will exactly have boundaries which are a sum of squares described by $E^\infty_{10}$ when lifted to $K$.
When given the family of 2-cycles, one can construct the corresponding 1-cycles by following the proof of Theorem~\ref{thm:order}.
Conversely, when given the family of 1-cycles, one can construct the 2-cycles by following the proof of Theorem~\ref{thm:order-nec}.

By re-centering the fundamental domain of $K$ we can also construct examples of spectral sequences where both toroidal and non-toroidal cycles fall \textit{outside} the zeroth column with multiplicity $O(n^2)$.
Indeed, the X-shaped complex shown in Figure~\ref{fig:intersecting_planes}(right), and its periodic copies (with an appropriate new cellular decomposition), provide a cover of $X_n$ by acyclic pieces, where \textit{every} $k$-cycle of $X_n$ will appear in the $k^\mathrm{th}$ column and zeroth row of $E^\infty$.


\section{Discussion and Future Work}
\label{sec:Discussion}

We have seen that the homology of periodic graphs can be entirely quantified by the information stored in a weighted quotient graph, and that this information is also conducive to quantifying the behaviour of finite windows with periodic boundary conditions.
However, the weights of a weighted quotient graph do not  generalise in an obvious way so we have no similar method to treat higher dimensional cellular complexes.

The difficulty in generalising vector weights arises by how $q(K)$ glues opposite sides of a unit cell together.
The boundary map and $q$ commute so a chain in $K$ can map to a cycle in $q(K)$ if the chain boundary cells are in the same translation equivalence class of $q(K)$.
Suppose though that we were to have some tool which in general could distinguish between Case~2 and Case~4 of Lemma~\ref{lem:four-cases} for cycles in $q(K)$.
Since we also care about translation invariance of $K$, such a tool \textit{must} satisfy the following axioms in order to uncouple this information:
\begin{enumerate}
    \item Evaluates differently on Cases~2 and 4 of Lemma~\ref{lem:four-cases}
    \item Encodes information about the boundaries of cells
    \item Translation invariant when lifted to $K$.
\end{enumerate}
Analogous to the weights of weighted quotient graphs, we may also require that this \textit{explicitly} encodes the relative offsets of boundary cells when lifted to $K$
\begin{enumerate}
  \setcounter{enumi}{3}
  \item Assigns a value in $T$ to each cell in $q(K)$
  \item Respects the linearity of the chain group
\end{enumerate}
A relatively simple candidate which satisfies these axioms is a $T$-valued cochain of $q(K)$ which lifts in a ``nice'' way to $K$.
This motivates the notion of a $k$-\textit{weighted quotient space} of $K$ as being $q(K)$ equipped with a $k$-\textit{weight} $\omega_k\in C^k(q(K);T)$ satisfying the following properties
\begin{enumerate}
    \item There exists $f_{k-1}\in C^{k-1}(K;T)$ such that $df_{k-1}=\omega_k\circ q$
    \item For any $\tilde{\gamma}\in Z_k(q(K))$, $\omega_k(\tilde{\gamma})=0$ if and only if $\tilde{\gamma}=q(\gamma)$ for some $\gamma\in Z_k(K)$.
\end{enumerate}
This definition solves the issue of uniquely pairing boundary cells by pairing all of them, at the cost of losing information of relative offsets of any individual pair of boundary cells.
For $k=0$, $q(K)$ will have a unique $0$-weight $\omega_0=0$.
For $k=1$ we have shown that connected weighted quotient graphs exactly satisfy this definition of $1$-weighted quotient spaces, where $f_0$ is defined on a vertex $v$ in $K$ by its translation from the representative of its translational equivalence class and extends $\Z$-linearly to $C_k(K)$.
However, for disconnected WQG's this may fail.
We conjecture that $k$-weighted quotient spaces of $K$ exist and are well-defined (possibly with minor additional restrictions) and leave it as an open question as to whether other constructions of $k$-weighted quotient spaces may exist in general.

If weighted quotient spaces exist, an interesting problem with this approach is also to generalise the concept of these weights to a broader class of spaces.
In Theorem~\ref{thm:graphH0} and Theorem~\ref{thm:graphH1} we appealed to covering space theory and the fact that $T$ acts freely on $K$ to relate edge weights in $q(K)$ to the homology of $K$, but we may question which, if any, of these conditions may be relaxed and, if so, by how much.
Is it perhaps possible to define such weights on $q(K)$ to determine the homology of $K$ when $q:K\to q(K)$ is no longer even a fibration, such as for the projection of a Rips complex onto its shadow?

We have also shown that the Mayer-Vietoris spectral sequence and scaling analysis provides a heuristic for identifying toroidal cycles which partially addresses how to recover higher-dimensional homology of a periodic cell complex.

To conclude, we return to the discussion in Section~\ref{sec:Introduction} on periodic point patterns.
The persistent homology of distance filtrations on periodic point patterns will be neither pointwise finite dimensional nor q-tame, meaning there may not necessarily be an associated persistence diagram.
Birth and death pairs will appear with infinite multiplicity if such pairs are well-defined.
Our method of employing the Mayer-Vietoris spectral sequence will not necessarily be affected by this as the analysis is done only on finite quotient spaces.
The heuristics used with the MVSS will generalise to persistence calculations, although there are other problems we will encounter with the persistence-MVSS.
Namely, the MVSS will change for different choices of $T$ and corresponding $q(K)$.
Even ignoring this issue, one must also solve the extension problem, whereby an interval $[a,b)$ in the $(p+q)$-dimensional persistence diagram may, for example, decompose into the intervals $[a,c)$ and $[c,b)$ in the persistence diagrams of $E^\infty_{pq}$ and $E^\infty_{p-1,q+1}$ respectively.

On the other hand, our analysis of periodic graphs with weighted quotient graphs should still extend to a non-tame notion of persistence.
Recall $[G:H]=|G/H|$ denotes the cardinality of the quotient group $G/H$.
Our perspective, motivated by the proofs of Theorem~\ref{thm:graphH0} and Theorem~\ref{thm:graphH1}, suggests:
\begin{itemize}
    \item The addition of a new connected component in $q(K)$ causes infinitely many 0-cycle births in $K$.
    \item If the addition of an edge $e$ to a connected component $Q\subset q(K)$ causes $[W_{Q\cup\{e\}}\,:\,W_Q]=\infty$ then infinitely many 0-cycles of $K$ die and if $\rank(W_Q)>0$, infinitely many 1-cycles are born..  
    \item If the addition of an edge $e$ to connected component $Q$ causes $1<[W_{Q\cup\{e\}}\,:\,W_Q]<\infty$ 
    \begin{itemize}
        \item[$\rightarrow$] If $\rank(W_Q)=d$ then finitely many 0-cycles die and infinitely many 1-cycles are born in $K$,
        \item[$\rightarrow$] Otherwise, infinitely many 0-cycles die and infinitely many 1-cycles are born.
    \end{itemize}
    \item If after the addition of an edge $e$ to connected component $Q$ we have $W_{Q\cup\{e\}}=W_Q$ then infinitely many 1-cycles are born in $K$.
    \item If the addition of an edge $e$ joins two connected components $Q_1,Q_2$ in the WQG
    \begin{itemize}
        \item[$\rightarrow$] If $[T\,:\,W_{Q_1}],[T\,:\,W_{Q_2}]<\infty$ then finitely many 0-cycles die in $K$ and infinitely many 1-cycles are born,
        \item[$\rightarrow$] Otherwise infinitely many 0-cycles die in $K$, and if  $\rank(W_{Q_1}),\rank(W_{Q_2})>0$, then infinitely many 1-cycles will be born in $K$.
    \end{itemize}
    
\end{itemize}


\section*{Acknowledgements}\label{ackref}
The authors would like to thank the anonymous referees for their very detailed and insightful feedback, particularly with regards to Theorems~\ref{thm:graphH1} and~\ref{thm:order-nec}.
The authors would also like to thank Primoz Skraba for many helpful discussions, particularly in relation to the statement and proofs of Theorem~\ref{thm:order} and Theorem~\ref{thm:order-nec}. These conversations commenced at the Hausdorff Research Institute for Mathematics special program in Applied and Computational Topology, September 2017, and V.R.\ gratefully acknowledges financial support from the University of Bonn for her visit there.
Much of the development of this material also benefited from many conversations with Martin Helmer during his time at the Australian National University.
A.O.~acknowledges the support of the Additional Funding Programme for Mathematical Sciences, delivered by EPSRC (EP/V521917/1) and the Heilbronn Institute for Mathematical Research.

\bibliographystyle{plain}  
\bibliography{refsdb}

\end{document}